\documentclass{article}
\usepackage{graphicx} 
\usepackage{cancel}
\usepackage{yhmath}
\usepackage{color}
\usepackage[utf8]{inputenc}
\usepackage{appendix}
\usepackage{amsmath}
\usepackage{amsthm}
\usepackage{amsfonts}
\usepackage{amssymb}
\usepackage{longtable}
\usepackage{cite}
\usepackage{url}
\usepackage{graphics}
\usepackage{subcaption}
\usepackage{epsfig}
\usepackage{geometry}
\usepackage{verbatim}
\usepackage{csquotes}

\title{Towards the classification of scattered binomials}
\author{Daniele Bartoli\thanks{Dipartimento di Matematica e Informatica, Universit\`a degli Studi di Perugia,  Perugia, Italy. daniele.bartoli@unipg.it},  Francesco Ghiandoni\thanks{Dipartimento di Matematica e Informatica, Universit\`a degli Studi di Perugia,  Perugia, Italy. francesco.ghiandoni@unipg.it},
Alessandro Giannoni\thanks{Dipartimento di Matematica e Applicazioni  ``R. Caccioppoli'', Università di Napoli Federico II, Napoli, Italy, alessandro.giannoni@unina.it},
Giuseppe Marino\thanks{Dipartimento di Matematica e Applicazioni  ``R. Caccioppoli'', Università di Napoli Federico II, Napoli, Italy, giuseppe.marino@unina.it}}
\date{}

\newtheorem{theorem}{Theorem}[section]

\newtheorem{remark}[theorem]{Remark}
\newtheorem{cor}[theorem]{Corollary}
\newtheorem{prop}[theorem]{Proposition}
\newtheorem{definition}[theorem]{Definition}

\def\F{\mathbb{F}}

\newcommand{\Tr}{\mathrm{Tr}}

\newcommand{\PG}{\mathrm{PG}}

\def \gammal{{\rm \Gamma L}}

\def\Fq3{{\mathbb F}_{q^3}}
\def\fq{{\mathbb F}_{q}}

\def\Fn{{\mathbb F}_{q^n}}
\def\q{{^{q}}}
\def\qq{^{q^2}}
\def\qqq{{^{q^3}}}

\begin{document}
\maketitle
\begin{abstract}
    Let \( q \) be a prime power and \( n \) an integer. An \( \mathbb{F}_q \)-linearized polynomial \( f \) is said to be scattered if it satisfies the condition that for all \( x, y \in \mathbb{F}_q^n \setminus \{ 0 \} \), whenever \( \frac{f(x)}{x} = \frac{f(y)}{y} \), it follows that \( \frac{x}{y} \in \mathbb{F}_q \).

In this paper, we focus on scattered binomials. Two families of scattered binomials are currently known: the one from Lunardon and Polverino (LP), given by 
$
f(x) = \delta x^{q^s} + x^{q^{n-s}},$
and the one from Csajbók, Marino, Polverino, and Zanella (CMPZ), given by 
$
f(x) = \delta x^{q^s} + x^{q^{s + n/2}},$
where \( n = 6 \) or \( n = 8 \).

Using algebraic varieties as a tool, we prove some necessary conditions for a binomial to be scattered. As a corollary, we obtain that when \( q \) is sufficiently large and \( n \) is prime, a binomial is scattered if and only if it is of the form (LP). Moreover we obtain a complete classification of scattered binomial in $\Fn$ when $n\leq8$ and $q$ is large enough.
\end{abstract}

\section{Introduction}\label{intro}

In this paper, \( q \) represents a prime power \( p^e \), where \( p \) is a prime. Let \( n \) be a positive integer. An \( \mathbb{F}_q \)-linearized polynomial (or simply an \( \mathbb{F}_q \)-polynomial) in \( \mathbb{F}_{q^n}[X] \) is a polynomial of the form
\[
f = \sum_{i=0}^d a_i X^{q^i}, \quad a_i \in \mathbb{F}_{q^n}, \quad i = 0, 1, \dots, d.
\]
If \( a_d \neq 0 \), the \( \mathbb{F}_q \)-polynomial \( f \) has \( q \)-degree \( d \). We can consider \( L_{n,q} \), the set of all \( \mathbb{F}_q \)-polynomials over \( \mathbb{F}_{q^n} \), and since we are interested in the evaluation of these polynomials in \( \mathbb{F}_q^n \), we can simply consider the quotient
\[
\mathcal{L}_{n,q} = L_{n,q} / (X^{q^n} - X).
\]
In this paper, we will consider \( f \) as an equivalence class in \( \mathcal{L}_{n,q} \).

Linearized polynomials have numerous applications in various fields, including combinatorics, coding theory, and cryptography. This paper focuses on a special type of \( \mathbb{F}_q \)-polynomial known as scattered polynomial, introduced by Sheekey in \cite{sheekey2016new} in the context of optimal codes in the rank metric. A scattered polynomial \( f \in \mathcal{L}_{n,q} \) satisfies the following condition for any \( x, y \in \mathbb{F}_{q^n}^* \):
\[
\frac{f(x)}{x} = \frac{f(y)}{y} \implies \frac{x}{y} \in \mathbb{F}_q.
\]

The graph of \( f \in \mathbb{F}_{q^n}[X] \) is defined as
\[
U_f = \{(x, f(x)) : x \in \mathbb{F}_{q^n} \}.
\]
If \( f \in \mathcal{L}_{n,q} \), then \( U_f \) is an \( \mathbb{F}_q \)-subspace of \( \mathbb{F}_{q^n}^2 \).

Two polynomials \( f \) and \( g \) in \( \mathcal{L}_{n,q} \) are called \( \mathrm{\Gamma L} \)-equivalent, or simply equivalent, if their graphs lie in the same orbit under the action of the group \( \mathrm{\Gamma L}(2,q^n) \).

So far, the known examples of \( \mathrm{\Gamma L} \)-inequivalent scattered polynomials are the following.
\begin{itemize}
    \item[(a)] \( f^1_s = x^{q^s} \in \mathbb{F}_q[x] \), \( 1 \leq s \leq n-1 \), \( \gcd(s, n) = 1 \), see \cite{blokhuis2000scattered, csajbok2016pseudoregulus}.
    \item[(b)] \( f^2_{s, \delta} = \delta x^{q^s} + x^{q^{n-s}} \in \mathbb{F}_q[x] \), \( n \geq 4 \), \( \text{N}_{q^n / q}(\delta) = \delta^{(q^n - 1) / (q - 1)} \notin \{0, 1\} \) (here \( \text{N}_{q^n / q}(\cdot) \) denotes the norm function from \( \mathbb{F}_{q^n} \) over \( \mathbb{F}_q \)), \( 1 \leq s \leq n-1 \), \( \gcd(s, n) = 1 \), \cite{LunardonPolverino2001, LavMarPolTro2015, sheekey2016new, lunardon2018generalized}.
    \item[(c)] \( f_{s, \delta}^3 = \delta x^{q^s} + x^{q^{s + n/2}} \in \mathbb{F}_q[x] \), \( n \in \{6, 8\} \), \( \text{N}_{q^n / q^{n/2}}(\delta) \notin \{0, 1\} \), \( 1 \leq s \leq n-1 \), \( \gcd(s, n/2) = 1 \), and certain choices of \( \delta \) and \( q \), \cite[Theorems 7.1 and 7.2]{csajbok2018new}. See also \cite{BartoliCsajbokMontanucci2021, PolverinoZullo2020rootsof...} for the case \( n = 6 \) and \cite{TimpanellaZini2024} for \( n = 8 \).
    \item[(d)] \( f_c^4 = x^q + x^{q^3} + cx^{q^5} \in \mathbb{F}_q[x] \), \( n = 6 \), \( q \) odd and \( c^2 + c = 1 \), or \( q \) even and certain choices of \( c \), see \cite{marino2020mrd, csajbok2018new2} for \( q \) odd and \cite{BartoliLongobardiMarinoTimpanella2024} for \( q \) even.
    \item[(e)] \( f_{h, t, \sigma}^5 = x^\sigma + x^{\sigma^{t-1}} - h^{1 - \sigma^{t+1}} x^{\sigma^{t+1}} + h^{1 - \sigma^{2t-1}} x^{\sigma^{2t-1}} \in \mathbb{F}_q[x] \), \( t \geq 3 \), \( n = 2t \), \( q \) odd, \( \sigma \) a generator of \( \text{Gal}(\mathbb{F}_{q^n} | \mathbb{F}_q) \), \( h^{q^t + 1} = -1 \), \cite{MR4173668, longobardi2023large, longobardi2021linear, NPZ}.
    \item[(f)] \( f_{m, t, \sigma}^6 = x^{\sigma^{t-1}} + x^{\sigma^{2t-1}} + m(x^\sigma - x^{\sigma^{t+1}}) \in \mathbb{F}_q[x] \), \( t > 4 \), \( q > 5 \), \( n = 2t \), and certain choices of \( m \in \mathbb{F}_{q^t} \), \cite{smaldore2024newscatteredlinearizedquadrinomials}.
\end{itemize}

In this paper, we will focus on scattered binomials, providing some necessary conditions for them to be scattered. We will use our results to show that when \( q \) is large enough and \( n \) is a prime number, a binomial is scattered if and only if it is of the type (b). 

The techniques used in this work rely on connections with specific algebraic varieties over finite fields. Notably, to overcome the challenge of studying overly complicated varieties, we analyze a projectively equivalent variety that is simpler to study. 

In particular, we determine sufficient conditions for their irreducibility which, thanks to the use of Lang-Weil type theorems, translate into sufficient conditions for proving that the binomials under investigation are not scattered, thus, in some cases, providing the desired classification result.

\section{Preliminaries}

\subsection{Scattered binomials}
As seen in Section \ref{intro}, the list of inequivalent $\F_q$-scattered subspaces of $\F_{q^n}^2$ defined as the graph of an $\F_q$-linearized scattered binomial consists of two families:

\begin{description}
    \item [(LP)] \[
        U_{s,\delta}^{2}=\{(x,\delta x^{q^s} + x^{q^{n-s}}): x \in \F_{q^n}\},
    \]
    where \( n\geq 4 \), \( \textnormal{N}_{q^n / q}(\delta)=\delta^{(q^n-1)/(q-1)}\notin \{0, 1\} \), \( 1\leq s \leq n-1 \), and \( \gcd(s,n)=1 \).
    
    \item [(CMPZ)] \[
        U_{s,\delta}^{3}=\{(x,\delta x^{q^s} + x^{q^{s+n/2}}):x \in \mathbb{F}_{q^n}\},
    \]
    where \( n \in \{6,8\} \), \( \textnormal{N}_{q^n /q^{n/2}}(\delta)=\delta^{1+q^{n/2}}\notin \{0,1\} \), \( 1\leq s \leq n-1 \), \( \gcd(s,n/2)=1 \), and for specific choices of \( \delta \) and \( q \).
\end{description}

The latter family was introduced by Csajbók, Marino, Polverino, and Zanella in \cite{csajbok2018new}, where they provided sufficient conditions on the coefficient \( \delta \) for the binomial to be scattered, both for \( q \) even and odd, and for \( n \in \{6,8\} \). Several papers have contributed partial results toward the classification of scattered linear sets of type
\[
L(U_{s,\delta}^{3,n})=\{\langle(x,\delta x^{q^s} + x^{q^{s+n/2}})\rangle_{\F_{q^n}}:x \in \mathbb{F}^*_{q^n}\} \subseteq \PG(1,q^n),
\]
for any even \( n \). 

It is not difficult to see that in \( \textnormal{PG}(1,q^4) \), the (LP) and (CMPZ) families overlap when \( \delta^{(q^n-1)/(q-1)}\neq 1 \). Moreover, by \cite[Theorem 3.4]{csajbokzanella2018PG1q4}, the linear sets they define are the only ones, up to \( \textnormal{GL}(2,q^4) \)-equivalence, that are maximum scattered in \( \textnormal{PG}(1,q^4) \). (Observe that for \( \delta=0 \), we obtain the linear sets of pseudo-regulus type.) This implies, in particular, that the only scattered binomials in \( \textnormal{PG}(1,q^4) \) belong to the (LP) family, introduced by Lunardon and Polverino in \cite{LunardonPolverino2001}.

The case \( n=6 \) was further investigated by Bartoli, Csajbók, and Montanucci in \cite{BartoliCsajbokMontanucci2021}, as well as by Polverino and Zullo in \cite{PolverinoZullo2020rootsof...} using different techniques. They ultimately characterized, in terms of the coefficient \( \delta \), which binomials are scattered. Furthermore, in \cite{BartoliCsajbokMontanucci2021}, the authors used these conditions on \( \delta \) to prove a conjecture proposed in \cite{csajbok2018new} concerning the number of new maximum scattered subspaces defining linear sets of type \( L(U_{s,\delta}^{3,n}) \).

\begin{theorem}\cite{BartoliCsajbokMontanucci2021},\cite[Theorem 7.3]{PolverinoZullo2020rootsof...}  \label{thm n=6 caratt scatt}
The $\fq$-linear set
\[
L(U_{s,\delta}^{3,6})=\{\langle(x,\delta x^{q^s} + x^{q^{s+3}})\rangle_{\F_{q^6}}:x \in \mathbb{F}^*_{q^6}\} \subseteq \PG(1,q^6),
\]
with \( \textnormal{N}_{q^6 / q^3}(\delta)=\delta^{1+q^3}\neq 1 \), is maximum scattered if and only if the equation
\begin{equation} \label{eq tr N n=6 caratt scatt}
    Y^2- (\Tr_{q^3/q}(A) -1)Y + \textnormal{N}_{q^3/q}(A) = 0,
\end{equation}
with \( A=\frac{\delta^{1+q^3}}{\delta^{1+q^3}-1} \), admits two roots over \( \fq \). In particular, such a \( \delta \) always exists for any \( q > 2 \). Additionally, there are exactly
\[
\left\lfloor \frac{(q^2+q+1)(q-2)}{2} \right\rfloor
\]
equivalence classes, under the action of \( \textnormal{GL}(2,q^6) \), of maximum scattered linear sets of type \( L(U_{1,\delta}^{3,6}) \).
\end{theorem}

For the case \( n=8 \), explicit sufficient conditions for obtaining maximum scattered linear sets have been provided in \cite[Theorem 7.2]{csajbok2018new}. Specifically, if \( q \) is odd and \( \textnormal{N}_{q^8/q^4}(\delta)=\delta^{1+q^4}=-1 \), then \( f_{s,\delta}=x^{q^s}+\delta x^{q^{s+4}} \) is scattered in \( \F_{q^8} \). Moreover, for \( q\leq 11 \) odd, such a condition on \( \delta \) is also necessary (see \cite[Remark 7.4]{csajbok2018new}). 

Motivated by this, Timpanella and Zini \cite{TimpanellaZini2024} investigated the asymptotic case, i.e., for large odd \( q \). Through a detailed analysis of certain \( 3 \)-dimensional \( \mathbb{F}_q \)-rational algebraic varieties in a \( 7 \)-dimensional projective space, they proved the following result.

\begin{theorem}\cite[Theorem 1.1]{TimpanellaZini2024} \label{thm n=8 timp zini}
Let \( q\geq 1039891 \) be odd, \( \delta \in \mathbb{F}^*_{q^8} \), and \( s \in \{1,3,5,7\} \). Then 
\( f_{s,\delta}=x^{q^s}+\delta x^{q^{s+4}} \) is scattered over \( \mathbb{F}_{q^8} \) if and only if  \( \delta^{1+q^4}=-1 \).
Moreover, the corresponding \( 2 \)-dimensional \( \F_{q^8} \)-linear MRD code 
$
\mathcal{C}_{s,\delta}=\langle x, f_{s,\delta} \rangle_{\F_{q^8}}
$
is inequivalent to all previously known MRD codes.
\end{theorem}




Results suggest that binomials of type (LP) are the only ones that remain scattered for infinitely many values of \( n \) and \( q \), in both the even and odd cases.

 In the following, to further support this conjecture, we will establish a classification result for scattered binomials, demonstrating the existence of infinitely many values of \( n \) for which, when \( q \) is sufficiently large, the only scattered binomials over \( \mathbb{F}_{q^n} \) belong to the (LP) family (see Corollary \ref{cor solo bin LP}). To achieve this, we reformulate the problem as the study of certain algebraic varieties of small degree relative to the size of the base field and apply Lang-Weil-type theorems (see Theorem \ref{thm lang weil versione tredici terzi}) to establish the existence of \( \Fn \)-rational points on these varieties.

\subsection{On the equivalence of scattered binomials}
Let us consider the non-degenerate symmetric bilinear form $\langle\cdot,\cdot\rangle$ of $\F_{q^n}$ over $\F_q$ defined by
$$\langle x,y \rangle = \Tr_{q^n/q}(xy),$$
where $\Tr_{q^n/q}(x) := x + x^q + \cdots + x^{q^{n-1}}$.

Given an $\F_q$-linearized polynomial $f(X) = \sum_{i=0}^{n-1} a_i x^{q^i}$, we can consider it as a map from $\F_{q^n}$ to itself. The adjoint map of $f(x)$ with respect to this bilinear form is $\hat{f}(x) = \sum_{i=0}^{n-1} a_i^{q^{n-i}} x^{q^{n-i}}$.   

\begin{theorem} \cite[Lemma 3.1]{Csajbokmarinopolverino2018classesofequivalence} \label{thm equivalenza aggiunto}
    Let $f$ be an $\F_q$-linearized polynomial in $\F_{q^n}$. Then $f$ is scattered if and only if $\hat{f}$ is scattered.
\end{theorem}

In this paper, we will study the scattered properties of binomials $f(x) = x^{q^I} + \alpha x^{q^J}$ in $\F_{q^n}$, with $1 \le I, J \le n - 1$, $I \notin \{J, n - J\}$. From the $\mathrm{\Gamma L}$-equivalence on polynomials, we can assume that $I < J$, and since studying the scattered properties of $f(x) = x^{q^I} + \alpha x^{q^J}$ is equivalent to studying $\hat{f}(x) = x^{q^{n - I}} + \alpha^{q^{n - J}} x^{q^{n - J}}$, we can also assume that $I < \frac{n}{2}$. We will use these assumptions in the following sections to simplify the case analysis.

In the case of (CMPZ) binomials, there is also this result.
\begin{prop}\cite[Proposition 5.1]{csajbok2018new} \label{prop equiv CMPZ somma eq n/2}
Two $\F_q$-subspaces $U^3_{s, \delta}$ and $U^3_{\overline{s}, \overline{\delta}}$ of $V = \F_{q^n} \times \F_{q^n}$ of the \textnormal{(CMPZ)} form with $\delta, \bar{\delta} \in \F_{q^n}^*$,  $\textnormal{N}_{q^n/q^{n/2}}(\delta) \ne 1$, $\textnormal{N}_{q^n/q^{n/2}}(\bar{\delta}) \ne 1$,  $1 \leq s, \bar{s} < n$ and $\gcd(n, s) = \gcd(n, \bar{s}) = 1$, are ${\rm \Gamma L}(2, q^n)$-equivalent if and only if either 
$$s = \bar{s} \mbox{\quad and \quad} \textnormal{N}_{q^n/q^n}(\bar{\delta}) = \textnormal{N}_{q^n/q^n}(\delta)^\sigma$$ 
or 
$$s + \bar{s} = n \mbox{\quad and \quad} \textnormal{N}_{q^n/q^{n/2}}(\bar{\delta}) \textnormal{N}_{q^n/q^{n/2}}(\delta)^\sigma = 1,$$
for some automorphism $\sigma \in \textnormal{Aut}(\F_{q^{n/2}})$.  
\end{prop}

\begin{remark} \label{rem norm no 1 and I < n/4}
When $\textnormal{N}_{q^n/q^{n/2}}(\delta) = 1$, it is trivial that $f(x) = \delta x^{q^s} + x^{q^{s + n/2}}$ is not scattered since $\dim_{\F_q}(\ker(f)) = n/2$. So, when studying the (CMPZ) binomials, we will assume $\textnormal{N}_{q^n/q^{n/2}}(\delta) \neq 1$, and using the latter proposition, we can also assume $I < \frac{n}{4}$.    
\end{remark}

Lastly, since $f(x) = x^{q^I} + \alpha x^{q^J}$ is $q^{\gcd(I, J, n)}$-linearized, we can also assume $\gcd(I, J, n) = 1$ to avoid trivial cases.

\subsection{Algebraic varieties}
Here we review fundamental concepts regarding algebraic varieties, directing readers to \cite{shafarevich1994basic} for an in-depth exploration of varieties. For a comprehensive understanding of how algebraic geometric methods apply to polynomials over finite fields, \cite{Bartoli:2020aa4} serves as a valuable resource. \\
Let $q$ be a prime power. We denote by $\overline{\fq}$ the algebraic closure of $\fq$ and by $\fq[X_0,\dots,X_{r} ]$ the ring of polynomials in $r+1$ variables with coefficients in $\fq.$ Moreover, $\mathbb{P}^r(\mathbb{K})$ and $\mathbb{A}^{r}(\mathbb{K})$ (or $\mathbb{K}^r$) denote the projective and the affine space of dimension $r\in \mathbb{N}$ over a field $\mathbb{K}$, respectively. A subset of $\mathbb{P}^{r}(\mathbb{K})$ is a (projective) algebraic variety  if it is the set of zeros of a finite number of 
 homogeneous polynomials  in  $\mathbb{K}[X_{0},\dots,X_{r}]$. The dimension of an algebraic variety $\mathcal{V}$ is the maximal
integer $s$ (that we denote by $\dim(\mathcal{V})$) for which there exists a strictly decreasing chain $\mathcal{V}_0 \supsetneqq \mathcal{V}_1 \supsetneqq \dots \mathcal{V}_{s}
 \supsetneqq \emptyset$ of length $s$ of irreducible subvarieties $\mathcal{V}_i \subset \mathcal{V}.$ If $\mathcal{V}\subseteq \mathbb{P}^r(\mathbb{K})$ is defined by $t$ equations, then $\dim(\mathcal{V})\geq r-t.$ 
 If a variety $\mathcal{V} \subseteq  \mathbb{P}^r(\mathbb{F}_q)$ is defined by
$F_i(X_0, \dots , X_r) = 0,$ for $i = 1, \dots  ,s,$ an $\mathbb{F}_q$-rational point of $\mathcal{V}$ is a point
$(x_0 : \dots : x_r) \in \mathbb{P}^r(\mathbb{F}_q)$ such that $F_i(x_0, \dots, x_r) = 0,$ for $i = 1, \dots , s.$  The set of the $\mathbb{F}_q$-rational points of $\mathcal{V}$ is usually
denoted by $\mathcal{V}(\mathbb{F}_q).$ If $\dim(\mathcal{V})=s,$ we say that $\mathcal{V}$ is a variety of degree $d$ (and write $\deg(\mathcal{V})=d$) if $d=\#(\mathcal{V}\cap H),$ where $H \subseteq\mathbb{P}^r(\overline{\mathbb{F}_q}) $ is a general projective subspace of dimension $r-s.$  To determine the degree of a variety is generally not straighforward; however an upper bound to $\deg(\mathcal{V)}$ is given by $\prod_{i=1}^{s}\deg(F_i).$ \\ 
If $s=1$ and $r=2,$ $\mathcal{V}$ is a (plane) curve and it is absolutely irreducible if the corresponding polynomial $F(X_0,X_1,X_2)$ is absolutely irreducible, i.e. it is irreducible in  $\overline{\mathbb{F}_q}[X_0,X_1,X_2].$  \\
A crucial point in our investigation of scattered binomials  is to prove the existence of suitable $\mathbb{F}_q$-rational points in algebraic varieties $\mathcal{V}$ attached to each binomial. In this direction, an estimate for the number of $\fq$-rational points of an  absolutely irreducible
variety over $\fq$ is provided by the Lang-Weil bound, which is a generalization to higher dimension of
the Hasse-Weil bound for curves. We will use the following refined version of Lang-Weil bound, due to Cafure and Matera.
\begin{theorem}\cite[Theorem 7.1]{MR2206396} \label{thm lang weil versione tredici terzi}
    Let $\mathcal{V} \subseteq \mathbb{P}^r(\mathbb{F}_q)$ be an absolutely irreducible variety
defined over $\mathbb{F}_q$ of dimension $n  > 0$ and degree $\delta,$ and not contained in the hyperplane at infinity. If $q > 2(n + 1)\delta^2,$ then the following
estimate holds:
\begin{equation}
  \big|\#(\mathcal{V}( \fq)\cap \mathbb{A}^r(\mathbb{F}_q)) - q^n
\big| \leq (\delta - 1)(\delta - 2)q^{n-1/2} + 5\delta^{13/3}q^{n-1}.  
\end{equation}
\end{theorem}

\subsection{Link with algebraic varieties} \label{sec link variety scatt}
Consider $\mathbb{F}_q$-linearized binomials of type $f(x)=x^{q^{I}}+\alpha x^{q^J}$ over $\mathbb{F}_{q^n},$ with $\alpha \in \mathbb{F}_{q^n}, 0 < I < J < n.$ 

The following link with between scattered polynomials and algebraic varieties over finite fields is well known; see \cite{bartoli2018exceptional}.
\begin{prop} \label{link with curves}
    The binomial $f(x)=x^{q^{I}}+\alpha x^{q^J}$ is scattered over $\mathbb{F}_{q^n}$ if and only if all the affine points of the curve
    $$ \mathcal{C}_{f}:F(X,Y):=\dfrac{X(Y^{q^I}+\alpha Y^{q^J})-Y(X^{q^I}+\alpha X^{q^J})}{X^qY-XY^q}=0 $$ lie on the lines $\lambda X +\mu Y=0$ for some $(\lambda,\mu) \in \mathbb{F}_q^2\setminus \{(0,0)\}$.
\end{prop}

Consider the curve 
$\mathcal{C}^{\prime}_{f}: \frac{F(XY,Y)}{Y^{q^I-q}} =0.$ There is a close relationship between the set of $\mathbb{F}_{q^n}$-rational points of such a curve and the  set of $\mathbb{F}_{q^n}$-rational points of $\mathcal{C}_f$ (note that the $\mathcal{C}^{\prime}_{f}$ is the geometric transform of $\mathcal{C}_f$ via the quadratic transformation $(x,y) \mapsto (\frac{x}{y},y)$; see \cite[Definition 4.41]{HKT}). 

This is highlighted in the following. 

\begin{prop} \label{link with curves 2}
    The binomial $f(x)=x^{q^{I}}+\alpha x^{q^J}$ is scattered over $\mathbb{F}_{q^n}$ if and only if all the affine points of the curve $\mathcal{C}^{\prime}_{f}$ lie on the lines $X +\mu =0,$ for some $\mu\in \mathbb{F}_q,$ or $Y=0$.
\end{prop}


We have that 

\begin{eqnarray*}
F(XY,Y)/Y^{q^I-q}&=&\dfrac{XY(Y^{q^I}+\alpha Y^{q^J})-Y(X^{q^I}Y^{q^I}+\alpha X^{q^J}Y^{q^J})}{Y^{q^I-q}(X^qY^{q+1}-XY^{q+1})} \\
&=&\dfrac{Y^{q^I+1}(X-X^{q^I})+\alpha Y^{q^J+1}(X-X^{q^J})}{Y^{q^I+1}(X^q-X)}\\
&=& \dfrac{X-X^{q^I}+\alpha Y^{q^J-q^I}(X-X^{q^J})}{X-X^q}.
\end{eqnarray*} 
Thus $\mathcal{C}^{\prime}_f$ reads 
$$ $$
\begin{equation}
    \alpha Y^{q^J-q^I}=-\dfrac{X^{q^I}-X}{X^{q^J}-X},
\end{equation}
which is equivalent, after replacing $Z=Y^{q^I},$ $\alpha \mapsto -\frac{1}{\alpha}$ and $K=J-I,$ to $Z^{q^K-1}-\alpha\frac{X^{q^I}-X}{X^{q^J}-X}=0$. In the following we will consider the  curve 
\begin{equation}\label{Eq:finale}
    \mathcal{C}^{\prime\prime}: Z^{q^K}-\alpha\dfrac{X^{q^I}-X}{X^{q^J}-X}Z=0,
\end{equation}
obtained adding the line $Z=0$ to the curve $Z^{q^K-1}-\alpha\frac{X^{q^I}-X}{X^{q^J}-X}=0$.

A general strategy at this point involves the investigation of absolutely irreducible $\mathbb{F}_{q^n}$-rational components of such a curve, together with the Hasse-Weil theorem, in order to provide a lower bound on the number of affine $\mathbb{F}_{q^n}$-rational points. This would allow, via Proposition \ref{link with curves 2}, the characterization of scattered binomials. Unfortunately, such a strategy works only when the degree of the curves (which depends linearly on $I$ and $J$) is small enough with respect to $n$ (usually roughly $\sqrt[4]{n}$).


To avoid this matter, we consider a link between the curve \eqref{Eq:finale} and a suitable variety in $\mathbb{P}^{2n-1}(\mathbb{F}_q)$ of small degree; in this direction  Lang-Weil type theorems  play a crucial role. 

Let $\left\lbrace 	\xi,\xi^q,\dots,\xi^{q^{n-1}} \right\rbrace $ be a normal basis of $\mathbb{F}_{q^n}$ over $\mathbb{F}_{q}.$ Set $X=x_0\xi + \dots +x_{n-1}\xi^{n-1}$ and $Z=y_0\xi + \dots +y_{n-1}\xi^{n-1},$ where $x_i,y_i \in \mathbb{F}_q.$

Thus, Equation \eqref{Eq:finale} reads $f_{\alpha}^{(0)}(x_0,\ldots,x_{n-1},y_0,\ldots,y_{n-1})=0,$ where $f_{\alpha}^{(0)}$ is defined as
\begin{eqnarray*}
    \left ((x_0\xi + \dots +x_{n-1}\xi^{n-1})^{q^J}-(x_0\xi + \dots +x_{n-1}\xi^{n-1})\right)(y_0\xi + \dots +y_{n-1}\xi^{n-1})^{q^K}\\
    -\alpha \left((x_0\xi + \dots +x_{n-1}\xi^{n-1})^{q^I}-(x_0\xi + \dots +x_{n-1}\xi^{n-1})\right)(y_0\xi + \dots +y_{n-1}\xi^{n-1}).
\end{eqnarray*}

Clearly, $\mathbb{F}_{q^n}$-rational points of $\mathcal{C}^{\prime\prime}$ correspond to $\mathbb{F}_q$-rational points of $f_\alpha^{(0)}(x_0,\ldots,x_{n-1},y_0,\ldots,y_{n-1})=0$ and satisfy 

$$
\begin{cases}
 f_\alpha^{(0)}(x_0,\ldots,x_{n-1},y_0,\ldots,y_{n-1})&=0\\ 
f_{\alpha^q}^{(0)}(x_1,\ldots,x_{n-1},x_0,y_1,\ldots,y_{n-1},y_0)&=0\\ 
\vdots\\
f_{\alpha^{q^{n-1}}}^{(0)}(x_{n-1},\ldots,x_{n-2},y_{n-1},\ldots,y_{n-2})&=0.
\end{cases}
$$

The above system defines an $\mathbb{F}_q$-rational variety $\mathcal{V}$ in $\mathbb{P}^{2n-1}(\mathbb{F}_q)$ of degree at most $2^n$. In order to apply Lang-Weil theorem and obtain a lower bound on the number of $\mathbb{F}_q$-rational points in $\mathcal{V}$, corresponding to $\mathbb{F}_{q^n}$-rational points in $\mathcal{C}^{\prime\prime}$ one has to prove the existence of suitable absolutely irreducible $\mathbb{F}_q$-rational components in $\mathcal{V}$.

To this end, we investigate a projectively equivalent variety $\mathcal{W}$ via the following change of variables, defined over $\mathbb{F}_{q^n}$:

\begin{eqnarray*}
x_0\xi + \dots +x_{n-1}\xi^{n-1} &\mapsto& X_0\\
(x_0\xi + \dots +x_{n-1}\xi^{n-1})^q &\mapsto& X_1\\
\vdots\\
(x_0\xi + \dots +x_{n-1}\xi^{n-1})^{q^{n-1}} &\mapsto& X_{n-1}\\
y_0\xi + \dots +y_{n-1}\xi^{n-1} &\mapsto& Y_0\\
(y_0\xi + \dots +y_{n-1}\xi^{n-1})^q &\mapsto& Y_1\\
\vdots\\
(y_0\xi + \dots +y_{n-1}\xi^{n-1})^{q^{n-1}} &\mapsto& Y_{n-1}.
\end{eqnarray*}

It is worth emphasising that $\mathcal{W}$ and $\mathcal{V}$ are projectively equivalent over $\mathbb{F}_{q^n}$ and thus the number of absolutely irreducible components together with their degrees are kept invariant by such a transformation, although the field of definition of such components may vary. Moreover, since $\mathcal{V}$ is defined over $\fq$, it follows that $\mathcal{W}$ is a variety defined over $\Fn,$ fixed by the collineation \begin{eqnarray*}
    \Psi_q :  \mathbb{P}^{2n-1}(\mathbb{F}_{q^n}) &\rightarrow& \mathbb{P}^{2n-1}(\mathbb{F}_{q^n}) \\     \left(u_0:\dots:u_{n-1}:v_0:\dots:v_{n-1}\right)&\mapsto& \left(u_{n-1}^{q^{n-1}}:\dots:u_{n-2}^{q^{n-1}}:v_{n-1}^{q^{n-1}}:\dots:v_{n-2}^{q^{n-1}}\right),
\end{eqnarray*} and 
$\fq$-rational points (resp. absolutely irreducible components) of $\mathcal{V}$ correspond to points (resp. absolutely irreducible components) of $\mathcal{W}$ which are fixed by $\Psi_q.$

By applying this projectivity, we find that the equations defining $\mathcal{W}$ are
\begin{equation*}
    \mathcal{W}:\begin{cases}
        Y_K=\alpha f_0(X)Y_0\\
        Y_{K+1}=\alpha^qf_1(X)Y_1\\
        \vdots\\
        Y_n=\alpha^{q^{n-K}}f_{n-K}(X)Y_{n-K}\\
        Y_0=\alpha^{q^{n-K+1}}f_{n-K+1}(X)Y_{n-K+1}\\
        \vdots\\
        Y_{K-1}=\alpha^{q^{n-1}}f_{n-1}(X)Y_{n-1},
    \end{cases}
\end{equation*}
where $K=J-I$, $f_i(X):=f(X_i,X_{I+i},X_{J+i})$ for $i=0,\dots,n-1$, with $f(u,v,w):=\dfrac{u-v}{u-w}$, and the indices of $X_i$ are considered modulo $n.$

\section{Structure of $\mathcal{W}$ in the case $\gcd(n,K)=1$} \label{sec struc f W}
Note that $\mathcal{W}$ contains as components the linear subspaces $\mathcal{X}:X_0=X_1=\dots=X_{n-1}$ and  $\mathcal{Y} : Y_0=Y_1=\cdots=Y_{n-1}=0,$ which are both fixed by $\Psi_q.$ Therefore, $\mathcal{W} \setminus \{\mathcal{X} \cup \mathcal{Y}\}$ is also fixed by $\Psi_q,$ and its absolute irreducibility is sufficient to ensure the existence of suitable $\Fn$-rational points on the curve $\mathcal{C}^{\prime}_{f},$ for $q$ large enough, thanks to the machinery described in Subsection \ref{sec link variety scatt}. Now we  describe  the equations defining $\mathcal{W}.$ \\
Set  $G:=\gcd(n,K)$ and $D:=\frac{n}{G}.$ The endomorphism $\gamma:t \mapsto K \cdot t$ of $(\mathbb{Z}_n,+)$  is a $(G-1)$-correspondence,  i.e. a homomorphism with kernel $\ker(\lambda)=\{0,D,2D,\dots,(G-1)D\}$ of size $G.$ Thus we can partition $\mathbb{Z}_n$ as  $\cup_{i=0}^{G-1}\{\ell K+i:\ell=0,\dots,D-1\},$ and
 we can rearrange the equations of $\mathcal{W}$ as blocks $S_i$, $i=0,\ldots,G-1$,
where 
\begin{equation}\label{eq blocchi s1...sg-1}
   S_0:\begin{cases}
    Y_K=g_0(X) Y_0\\
    Y_{2K}=g_K(X) Y_K\\
    \vdots\\
    Y_{(D-1)K}=g_{(D-2)K}(X)Y_{(D-2)K} \\
    Y_{0}=g_{(D-1)K}(X)Y_{(D-1)K}
    \end{cases} \!\!\!\!\!\!  \dots  \ \ \ S_{G-1}:
    \begin{cases}
    Y_{K+G-1}=g_{G-1}(X) Y_{G-1}\\
    Y_{2K+G-1}=g_{K+G-1}(X) Y_{K+G-1}\\
    \vdots\\
    Y_{(D-1)K+G-1}=g_{(D-2)K+G-1}(X)Y_{(D-2)K+G-1}\\
    Y_{G-1}=g_{(D-1)K+G-1}(X)Y_{(D-1)K+G-1},
    \end{cases}
\end{equation} 
and $g_i(X)=\alpha^{q^i}f_i(X),$ for $i=0\dots,G-1$ and $j=0\dots,n-1.$\\
Note that the indices of all variables are considered modulo $n$, i.e., $X_s = X_{s\! \pmod{n}}$, and $Y_t = Y_{t\! \pmod{n}}.$ 
Thus, for each $i \in {0,\dots, G-1},$ one can replace $Y_{K+i}$ in the second equation of $S_i,$ $Y_{2K+i}$ in the third equation, $\dots,$ and $Y_{(D-1)K+i}$ in the last equation of $S_i$. After all these $n-1$ substitutions, the last equation of $S_i$ becomes
\begin{equation*}
    Y_i=\text{N}_{q^n/q^G}(\alpha)^{q^i} \prod_{\ell=0}^{D-1} f_{\ell K+i}(X)Y_i,
\end{equation*}
i.e., apart from $Y_i=0$ (from which we obtain the component $\mathcal{Y}),$
\begin{equation} \label{eq sist G pol irred W - Y}
    \text{N}_{q^n/q^G}(\alpha)^{q^i} \prod_{\ell=0}^{D-1} (X_{\ell K+i}-X_{\ell K +i +I})=\prod_{\ell=0}^{D-1} (X_{\ell K+i}-X_{\ell K +i +J}).
\end{equation}
Therefore, the absolute irreducibility of $\mathcal{W} \setminus \{\mathcal{X} \cup \mathcal{Y}\}$ is strongly related to the absolute irreducibility of the variety $\mathcal{Z}$ defined by the Equations \ref{eq sist G pol irred W - Y}, for $i=0,\dots,G-1.$
\begin{prop}
\label{prop W meno x y ass irr iff Z ass irr}    The variety $\mathcal{W} \setminus \{\mathcal{X} \cup \mathcal{Y}\}$ is  absolutely irreducible  if and only if $\mathcal{Z}$ is absolutely irreducible.
\end{prop}
In general, studying the absolute irreducibility of algebraic varieties of this kind is not straightforward. However, in some significant cases, it is possible to reduce everything to the study of a single polynomial, thus obtaining good characterizations of the entire starting variety. With the aim of finding infinitely many $n$ for which it is possible to fully characterize scattered binomials over $\Fn$, we focus on the  case $\gcd(n,K)=1$.

When $n$ and $K$ are coprime, the function $t\mapsto K\cdot t$ is a permutation of $\mathbb{Z}_n$ and thus we can rearrange the equations of $\mathcal{W}$ in the following way
\begin{equation}
W:\begin{cases}
    Y_K=\alpha f_0(X)Y_0\\
    Y_{2K}=\alpha^{q^K} f_K(X)Y_K\\
    Y_{3K}=\alpha^{q^{2K}} f_{2K}(X)Y_{2K}\\
    \vdots\\
    Y_{(n-2)K}=\alpha^{q^{(n-3)K}} f_{(n-3)K}(X)Y_{(n-3)K}\\
    Y_{(n-1)K}=\alpha^{q^{(n-2)K}} f_{(n-2)K}(X)Y_{(n-2)K}\\
    Y_0=\alpha^{q^{(n-1)K}} f_{(n-1)K}(X)Y_{(n-1)K},
\end{cases}
\end{equation}
where the indices are considered modulo $n$. 
Thus, combining each equation with the subsequent one obtains 
\begin{equation} \label{eq sist caso gcd 1 abs irr W iff P}
\begin{cases}
    Y_K=\alpha f_0(X)Y_0\\
    Y_{2K}=\alpha^{q^K+1} f_K(X)f_0(X)Y_0\\
    Y_{3K}=\alpha^{q^{2K}+q^K+1} f_{2K}(X)f_K(X)f_0(X)Y_0\\
    \vdots\\
    Y_{(n-2)K}=\alpha^{q^{(n-3)K}+\dots+q^K+1} f_{(n-3)K}(X)\dots f_K(X)f_0(X)Y_0\\
    Y_{(n-1)K}=\alpha^{q^{(n-2)K}+\dots+q^K+1} f_{(n-2)K}(X)\dots f_K(X)f_0(X)Y_0\\
    Y_0=\textnormal{N}_{q^n/q}(\alpha) f_{(n-1)K}(X)\dots f_K(X)f_0(X)Y_0.
\end{cases}
\end{equation}
Therefore, the variety $\mathcal{W} \setminus \{\mathcal{X} \cup \mathcal{Y}\}$ is absolutely irreducible if and only if  the  polynomial
\begin{equation}\label{eq p las of system f1f2}
    P(X):=AF_1(X)-F_2(X),
\end{equation}
  with $A=\textnormal{N}_{q^n/q}(\alpha)$, $F_1(X)=\prod_{\ell=0}^{n-1}(X_\ell-X_{I+\ell})$ and $F_2(X)=\prod_{\ell=0}^{n-1}(X_\ell-X_{J+\ell})$, is absolutely irreducible (all the indices are considered modulo $n$).

  \begin{remark} \label{rem produttore uguali iff I+J=n}
      We observe that  
$$\displaystyle\prod_{\ell=0}^{n-1}(X_\ell-X_{I+\ell})=\displaystyle\prod_{\ell=0}^{n-1}(X_\ell-X_{J+\ell}) \iff I+J=n,$$  
that is, \( f = x^{q^I}+\alpha x^{q^{n-I}} \) is an (LP)-type binomial. Moreover, for \( I+J=n \) and \( \text{N}_{q^n/q}(\alpha)\neq 1 \), the polynomial \( P(X) \) splits into \( n \) factors (and \( \mathcal{W} \setminus \{\mathcal{X} \cup \mathcal{Y}\} \) decomposes into \( n \) components), and it is the zero polynomial if and only if \( I+J=n \) and \( \text{N}_{q^n/q}(\alpha)=1 \). 

This is consistent with \cite[Theorem 3.4]{zanella2019condition}, where it is proven that \( f = x^{q^I}+\alpha x^{q^{n-I}} \) is not scattered for \( \text{N}_{q^n/q}(\alpha)=1 \).

  \end{remark}
  \subsection{Absolute irreducibility of $P(X)$}
  Firstly we note that  $n-J\neq I\neq J$ implies $\gcd(F_1,F_2)=1$. Thus we can write $F_1(X)=(X_0-X_I)(X_{n-I}-X_0)G(X_1,\dots,X_{n-1})$ and $F_2(X)=(X_0-X_J)(X_{n-J}-X_0)H(X_1,\dots,X_{n-1})$, so that 
   \begin{eqnarray} \label{eq P(x)}
      P(X)&=&A(X_0-X_I)(X_{n-I}-X_0)G-(X_0-X_J)(X_{n-J}-X_0)H  \nonumber\\ 
      &=&A(-GX_0^2+G(X_I+X_{n-I})X_0-GX_IX_{n-I})-(-HX_0^2+H(X_J+X_{n-J})X_0-HX_JX_{n-J}) \nonumber \\
      &=& X_0^2(H-AG)+X_0(AG(X_I+X_{n-I})-H(X_J+X_{n-J})) +HX_JX_{n-J}-AGX_IX_{n-I} \nonumber \\
      &=&aX_0^2+bX_0+c,   
  \end{eqnarray}
  with 
\begin{eqnarray*}
    a&=&H-AG\\
    b&=&AG(X_I+X_{n-I})-H(X_J+X_{n-J})\\
    c&=&HX_JX_{n-J}-AGX_IX_{n-I}.
\end{eqnarray*}
\begin{prop}
    The polynomials $a,b,c$ have no non-trivial factor in common.
\end{prop}
\begin{proof}
Note that \( \gcd(H,G) \neq 1 \) if and only if \( I = J \) or \( I = n - J \). So in our case, \( \gcd(H,G) = 1 \). Let us prove that \( \deg(\gcd(a,b,c)) = 0 \).

By contradiction, let \( d(X_1, \dots, X_{n-1}) = \gcd(a,b,c) \) with degree at least 1. Since \( d \mid a \) and \( d \mid b \), we have
$$ d \mid \left( b + (X_J + X_{n-J})a \right) = AG(X_I + X_{n-I} - X_J - X_{n-J}). $$

Analogously, since \( d \mid a \) and \( d \mid c \), we have
$$ d \mid \left( c - (X_JX_{n-J})a \right) = AG(X_JX_{n-J} - X_IX_{n-I}). $$

Thus,
$$ d \mid \gcd\left( AG(X_I + X_{n-I} - X_J - X_{n-J}), AG(X_JX_{n-J} - X_IX_{n-I}) \right) = AG. $$

This is a contradiction because from \( d \mid a \) and \( d \mid G \), we obtain \( d \mid H \), and so \( d \mid \gcd(G, H) = 1 \).

\end{proof}

In order to prove the absolute irreducibility of $P(X),$ we will make use of the concept of multiplicity of roots in a multivariate polynomial.

\begin{definition}
    Let $p(X_1,\dots,X_m)\in\overline{\mathbb{F}_q}[X_1,\dots,X_m]$ and $Q=(u_1,\dots,u_m)\in\overline{\mathbb{F}_q}^m$.
    Let 
    $$p(X_1+u_1,\dots,X_m+u_m)= p_k(X_1,\dots,X_m)+p_{k+1}(X_1,\dots,X_m)+\cdots,$$
    where each $p_i$ is either $0$ or homogeneous of degree $i$ and $p_k\neq 0$. We define the multiplicity $m_Q(p)$ of $p$ in $Q$ as $k$.
\end{definition}
We can extend the above definition to any element of $\overline{\mathbb{F}_q}(X_1,\dots,X_m)$ as follows
\begin{equation*}
    m_Q\left(\frac{p_1}{p_2}\right):=m_Q({p_1}{})-m_Q({p_2}).
\end{equation*}

\begin{prop}
    The following properties of the multiplicity hold true:
    \begin{enumerate}
        \item $m_Q(p_1p_2)=m_Q(p_1)+m_Q(p_2)$;
        \item $m_Q(p_1+p_2)\geq \min(m_Q(p_1),m_Q(p_2))$ and a sufficient condition for the equality to holds is $m_Q(p_1)\neq m_Q(p_2)$.
    \end{enumerate}
    
\end{prop}
\begin{remark}\label{remMultPol}
    Let $p(X_1,\dots,X_m)=a(X_2,\dots,X_m)X_1^2+b(X_2,\dots,X_m)X_1+c(X_2,\dots,X_m)\in\overline{\mathbb{F}_q}[X_1,\dots,X_m]$.
    
    If $q$ is odd then $p$ has a factor of the type $(a_1(X_2,\dots,X_m)X_1+b_1(X_2,\dots,X_m))$ if and only if $\Delta:=b^2-4ac$ is a square in $\overline{\mathbb{F}_q}[X_2,\dots,X_m]$.
Obviously if there exists $Q\in\overline{\mathbb{F}_q}^{m-1}$ such that $m_Q(\Delta)$ is odd, then $\Delta$ is not a square.

    If $q$ is even then $p$ has a factor of the type $(a_1(X_2,\dots,X_m)X_1+b_1(X_2,\dots,X_m))$ if and only if there exists $\omega\in\overline{\mathbb{F}_q}(X_2,\dots,X_m)$ such that $$\frac{ac}{b^2} = \omega^2+\omega.$$ 
    If there exists $Q\in\overline{\mathbb{F}_q}^{m-1}$ such that $m_Q(ac)$ is odd, and $2m_Q(b)>m_Q(ac)$ then this $\omega$ cannot exists, since we will have
$$0>m_Q(ac)-2m_Q(b)=m_Q(\omega^2+\omega)=m_Q(\omega^2)=2m_Q(\omega)$$ and this is a contradiction.
So for both $q$ odd and $q$ even, a sufficient condition for not having a factor of the type $(a_1(X_2,\dots,X_m)X_1+b_1(X_2,\dots,X_m))$ is the existence of $Q\in\overline{\mathbb{F}_q}^{m-1}$ such that $m_Q(ac)$ is odd, and $2m_Q(b)>m_Q(ac)$.
\end{remark}

\begin{remark} \label{rem irr se e solo se polo disp o delta no square}
Let $P(X)$ defined as in Equation \eqref{eq P(x)}. 
   By Remark \ref{remMultPol}, to prove the irreducibility of $P(X)$ it is sufficient to prove 
the existence of an $(n-1)$-tuple $Q=(u_1,\dots,u_{n-1})\in \overline{\mathbb{F}_q}^{n-1}$ such that $m_Q(ac)$ is odd, and $2m_Q(b)>m_Q(ac)$.
\end{remark}

In what follows, we will consider the lower indices of the variables modulo \(n\).
\begin{theorem} \label{prop P abs irr q odd}
    Let $n\geq5$. Then there exist an $(n-1)$-tuple $Q=(u_1,\dots,u_{n-1})\in \overline{\mathbb{F}_q}^{n-1}$ which is a zero of odd multiplicity $m_1$ for the polynomial $ac$ and a zero of multiplicity $m_2$ for $b,$ with $2m_2>m_1$. In particular $P(X)$ is absolutely irreducible.
\end{theorem}
\begin{proof}

Firstly let us consider $n\geq 7$.

We distinguish the following cases.

-\textbf{Case 1.}  $X_J-X_{n-J} \nmid H,\,X_J-X_{n-J} \nmid G$

Note that  $n\geq 7$ yields the existence of $\ell\in\{1,\dots,n-1\}$ such that $\ell,(\ell+I \mod n) \notin\{0,J,n-J\}.$ Consider $u_{J}=u_{n-J}=0,$ $u_\ell=\xi^{q^\ell}=u_{\ell+I}$, and $u_j=\xi^{q^j}$ for each $j \in \{1,\dots,n-1\} \setminus \{J,n-J,\ell,(\ell+I \mod n)\}.$ It follows that $a(Q)\neq 0,$ $b(Q)=0$, and $Q$ is a simple zero of $c,$ so $Q$ has the required properties.

-\textbf{Case 2.}  $X_J-X_{n-J} \nmid H,\,X_J-X_{n-J} \mid G$

 Consider $u_{J}=u_{n-J}=0,$ and $u_j=\xi^{q^j}$ for each $j \in \{1,\dots,n-1\} \setminus \{J,n-J\}.$ It follows that $a(Q)\neq 0,$ $b(Q)=0$, and $Q$ is a simple zero of $c,$ so $Q$ has the required properties.

-\textbf{Case 3.}  $X_J-X_{n-J} \mid H$

Note that $X_J-X_{n-J} \mid H$ implies $J=\frac{n}{3}$,  $X_I-X_{n-I} \nmid H$, and $X_I-X_{n-I} \nmid G$ since both the last two items of them would imply $I=\frac{n}{3}$, a contradiction. The condition $n\geq 7$ yields the existence of $\ell\in\{1,\dots,n-1\}$ such that $\ell,(\ell+J \mod n) \notin\{0,I,n-I\}$. Consider $u_{I}=u_{n-I}=0,$ $u_\ell=\xi^{q^\ell}=u_{\ell+J}$, and $u_j=\xi^{q^j}$ for each $j \in \{1,\dots,n-1\} \setminus \{I,n-I,\ell,(\ell+J \mod n)\}.$ It follows that $a(Q)\neq 0,$ $b(Q)=0$, and $Q$ is a simple zero of $c,$ so $Q$ has the required properties.

\vspace{0.4cm}

If $n=5$  we have to consider the cases $(I,J)=(1,2),(1,3),(2,3)$.

-$(I,J)=(1,2),(2,3)$

In these cases $X_J-X_{n-J}\mid G$, and we can proceed as in \textbf{Case 2}. 

-$(I,J)=(2,3)$

Note that in this case we can choose $\ell=4$, and obtain that $\ell,(\ell+I \mod n) \notin\{0,J,n-J\}.$ Then we can proceed as in \textbf{Case 1}.

\vspace{0.4cm}

When $n=6$ the cases to be considered are $(I,J)=(1,2),(2,3)$.

-$(I,J)=(1,2)$

In this case $X_J-X_{n-J}\mid H$, and we can choose $\ell=1$ to obtain that $\ell,(\ell+J \mod n) \notin\{0,I,n-I\}.$ Then we can proceed as in \textbf{Case 3}.

-$(I,J)=(2,3)$

Note that in this case we can choose $\ell=2$, and obtain that $\ell,(\ell+I \mod n) \notin\{0,J,n-J\}.$ Then we can proceed as in \textbf{Case 1}.
\end{proof}

\subsection{Main result}
 Now we are in position to prove our main result.
\begin{theorem} \label{main theorem}
       Let $n\geq 5,$ $\gcd(J-I,n)=1,$  $q \geq2^{\frac{13}{3}n+4}$ and $\alpha \in \mathbb{F}_{q^n}\!\setminus \! \{0\}.$  Then the  binominal
    $x^{q^I}+\alpha x^{q^J}$ is scattered on $\mathbb{F}_{q^n}$ if and only if  $I+J=n$ and $\textnormal{N}_{q^n / q}(\alpha) \neq 1.$   
   \end{theorem}  
\begin{proof}
 Consider $P(X_0,\dots,X_n)$ as in Equation \eqref{eq p las of system f1f2}. From Theorem \ref{prop P abs irr q odd}  we deduce the absolute irreducibility of $P(X_0,\dots,X_n).$ Then Equation \eqref{eq sist caso gcd 1 abs irr W iff P} implies that also $\mathcal{W}\setminus \{\mathcal{X}\cup \mathcal{Y}\}$ is absolutely irreducible, and it corresponds to an $\fq$-rational absolutely irreducible component of the variety $\mathcal{V},$ introduced in Section \ref{sec link variety scatt}, not contained in the linear subspaces $\overline{\mathcal{X}}:x_0=x_1=\dots=x_{n-1}$ and $\overline{\mathcal{Y}}:y_0=y_1=\dots=y_{n-1}=0.$ Moreover, from $\deg(\mathcal{V})\leq 2^n,$ it follows that such a component intersects $\overline{\mathcal{X}}\cup \overline{\mathcal{Y}}$ in at most $2^n+2^n=2^{n+1}$ $\fq$-rational points.  Theorem \ref{thm lang weil versione tredici terzi} yields \begin{equation*}
   \#(\mathcal{V} (\mathbb{F}_q)\cap \mathbb{A}^{2n-1}(\fq)) \geq q^n
 -(2^n-1)(2^n-2)q^{n-1/2} - 5\cdot 2^{13n/3}q^{n-1} > 2^{n+1},
\end{equation*}
when $q \geq 2^{\frac{13}{3}n+4}.$ Therefore we deduce the existence of $\fq$-rational points on $\mathcal{V}$ off $\overline{\mathcal{X}}\cup \overline{\mathcal{Y}}$ which correspond to $\Fn$-rational points $Q=(x,y)$ of $\mathcal{C}^{\prime}_{f}$ with $x \notin \fq$ and $y \neq 0.$ The statement follows from Proposition \ref{link with curves 2}. 
\end{proof}
\begin{cor} \label{cor solo bin LP}
           Let $n\geq 5$ be a prime number and $q\geq 2^{\frac{13}{3}n+4}$ be a prime power. Then the only scattered binomials over $\Fn$ are of (LP)-type. 
       \end{cor}

\section{Classification of scattered binomials for $n \leq 8$} \label{sect
class bin n fino a 8}
The aim of this section is to provide a complete characterization of scattered binomials over $\F_{q^n}$ for $n \leq 8$. For $n\in\{3,4\}$, according to 
{\cite[Section 4]{Csajbokmarinopolverino2018classesofequivalence}} 
and \cite[Theorem 3.4]{csajbokzanella2018PG1q4}, we already know that a scattered polynomial is $\gammal$-equivalent to either a pseudoregulus type polynomial or an (LP)-type polynomial for $n=4.$ Furthermore, from Corollary \ref{cor solo bin LP}, we can infer that over $\F_{q^n}$ there are no scattered binomials but only (LP) polynomials for $n\in\{5,7\}.$ Therefore, taking into account \cite[Theorem 1.1]{BartoliCsajbokMontanucci2021} and \cite[Theorem 1.1]{TimpanellaZini2024}, the pending cases to be considered are $f(x)=x^{q^I}+\alpha x^{q^J}$ over $\F_{q^n},$ where
\begin{itemize}
    \item $n=6$ and $G=\gcd(n,J-I)=2;$
    \item $n=8$ and $G=\gcd(n,J-I)=2;$
    \item $q$ even, $n=8,$ and $G=\gcd(n,J-I)=4.$
\end{itemize}
We will discuss each case separately, by means of deep analysis of the variety \begin{equation} \label{eq Z per gcd >1}\mathcal{Z}: \text{N}_{q^n/q^G}(\alpha)^{q^i} \prod_{\ell=0}^{D-1} (X_{\ell K+i}-X_{\ell K +i +I})=\prod_{\ell=0}^{D-1} (X_{\ell K+i}-X_{\ell K +i +J}) \ \ \ \ \ \ i=0,\dots,G-1,\end{equation} introduced in Section \ref{sec struc f W}; see Equation \eqref{eq sist G pol irred W - Y}.
\subsection{Case $n=6,$ $\gcd(n,J-I)=2$}\label{n6G2} 
By Theorem \ref{thm equivalenza aggiunto}  we can assume without restriction $I=1$ and $J=3.$\\
By Equations  \eqref{eq blocchi s1...sg-1}, \eqref{eq sist G pol irred W - Y},  $\mathcal{Z}$ reads 
\begin{equation}\label{sist sole x caso gcd 2}
    \mathcal{Z}:\begin{cases}
    (X_0-X_1)(X_2-X_3)(X_4-X_5)-\beta(X_0-X_3)(X_2-X_5)(X_4-X_1)=0 \\
    (X_1-X_2)(X_3-X_4)(X_5-X_0)-\beta^q(X_1-X_4)(X_3-X_0)(X_5-X_2)=0,
\end{cases}
\end{equation}
with $\beta=\alpha^{1+q^2+q^4}\in\F_{q^2}\setminus \{0\}.$
Solving for \( X_0 \) from the first equation of System~(\ref{sist sole x caso gcd 2}) and substituting it into the next equation, we obtain
\begin{equation}\label{eq pol secondo grado gcd 2 n 6}
    aX_1^2+bX_1+c=0,
\end{equation} where \begin{eqnarray}
    a&=&(1-\beta)X_3^2X_5 + (\beta+\beta^q-1)X_3X_4X_5 + (\beta-\beta^q)X_3X_5^2 - \beta X_4X_5^2   - X_3^2X_4  + X_3X_4^2 \nonumber\\&&+[\beta X_3^2 - (\beta+\beta^q-1)X_3X_4  + (\beta-\beta^q+1)(X_4-X_3)X_5    + \beta^qX_5^2   - X_4^2 ] \cdot X_2 \nonumber\\&&+ \beta^q(X_4-X_5)\cdot X_2^2\nonumber\\
    &=&a_0+a_1X_2+a_2X_2^2;
    \\b&=&X_5[(\beta-\beta^q+1)(X_3-X_5)X_3X_4 - (\beta+\beta^q+1)X_3X_4^2  + \beta X_4^2X_5  + (\beta^q-1)X_3^2X_5    ] \nonumber\\&&+(X_5-X_4)[(\beta-\beta^q-1)(X_3^2+X_4X_5) - (\beta+\beta^q-1)(X_4 +X_5)X_3 ]\cdot X_2 \nonumber \\&&
     +[(\beta-\beta^q-1)(X_3-X_5)X_4 + (\beta+\beta^q+1)X_3X_5    - (\beta^q-1)X_4^2-\beta X_3^2] \cdot X_2^2\nonumber\\
    &=&b_0+b_1X_2+b_2X_2^2;\\
     c&=&\beta^qX_3^2X_4X_5(X_4 - X_5)+[(\beta-\beta^q+1)(X_4-X_3)X_3X_4X_5 + (\beta+\beta^q-1)X_3X_4X_5^2\nonumber\\&& - \beta X_4^2X_5^2 - \beta^qX_3^2X_4^2 +X_3^2X_5^2]\cdot X_2  +[ \beta X_3^2X_4 - (\beta-\beta^q)X_3X_4^2 \nonumber\\&& -( \beta+\beta^q-1)X_3X_4X_5 + \beta X_4^2X_5    - X_3X_5^2 - X_4^2X_5 + X_4X_5^2]\cdot X_2^2 \nonumber\\
    &=&c_0+c_1X_2+c_2X_2^2.  
\end{eqnarray}

\begin{prop}\label{prop_coprime}
   The polynomials  $a,b,c$  have no non-constant common factors.
\end{prop}
\begin{proof}
    
    By contradiction let $d(X_2,X_3,X_4,X_5):=\gcd(a,b,c)$ be of degree at least $1$. Then $d\mid F_{a,b}:=a_2b-b_2a$ and $d\mid F_{a,c}:=a_2c-c_2a$.
    We have
    $$F_{a,b}=s_0(X_3,X_4,X_5)+s_1(X_3,X_4,X_5)X_2,\,\,F_{a,c}=t_0(X_3,X_4,X_5)+t_1(X_3,X_4,X_5)X_2,$$
    so $d\mid s_1F_{a,c}-t_1F_{a,b}\not\equiv0$. This means that $d$ does not depend on $X_2$, and  $d\mid a$ yields $d\mid a_2$, so we can assume $d=X_4-X_5$ up to a scalar multiple. It is obvious that $(X_4-X_5)\nmid a$, so we have a contradiction. 
\end{proof}

By Remark \ref{rem irr se e solo se polo disp o delta no square}, to obtain the irreducibility of $aX_1^2+bX_1+c$ in Equation \eqref{eq pol secondo grado gcd 2 n 6}, it is sufficient to provide the existence of a $4$-tuple $Q=(u_2,u_3,u_4,u_5)\in \overline{\mathbb{F}_q}^{4}$ which is a zero of odd multiplicity $m_1$ for the polynomial $ac$ and a zero of multiplicity $m_2$ for the polynomial $b$ such that $2m_2>m_1$. 
\begin{theorem} \label{thm punto ad ok gcd=2 n 6}
    The $4$-tuple $Q=(0,\xi,\xi^q,0)$ is a simple zero of $c$, a zero of $b$, and it is not a zero of $a$.
\end{theorem}
\begin{proof}
We can rewrite $c$ as 
\begin{equation*}
    X_2A_c(X_2,X_3,X_4)+X_5B_c(X_2,X_3,X_4,X_5),
\end{equation*}
where $X_2\nmid A_c$ and $X_2 \nmid B_c$. It is readily seen that  $Q$ is a simple point of $c$.

By straightforward computations,  $a(Q)\neq0$ and $b(Q)=0$. The claim follows.
\end{proof}

\begin{cor} \label{cor Z irr gcd 2 n 6}
    Let $\alpha \in \F_{q^6}\setminus \{0\},$ $\beta=\alpha^{1+q^2+q^4}\in \F_{q^2},$ and let $\mathcal{Z}$ be the variety defined by System \eqref{sist sole x caso gcd 2}. Then $\mathcal{Z}$ is absolutely irreducible.
\end{cor}
\begin{proof}
    The claim follows by  Proposition \ref{prop_coprime} and Theorem \ref{thm punto ad ok gcd=2 n 6}.
\end{proof}

\subsection{Case $n=8,$ $\gcd(n,J-I)=2$}
As before, it is not restrictive to assume $I=1$ and $J=3.$\\
Consider  the variety
\begin{equation}\label{sist sole x caso gcd 2 n=8}
\mathcal{Z}:\begin{cases}
    (X_0-X_1)(X_2-X_3)(X_4-X_5)(X_6-X_7)-\beta(X_0-X_3)(X_2-X_5)(X_4-X_7)(X_6-X_1)=0 \\
    (X_1-X_2)(X_3-X_4)(X_5-X_6)(X_7-X_0)-\beta^q(X_1-X_4)(X_3-X_6)(X_5-X_0)(X_7-X_2)=0,
\end{cases}
\end{equation}
with $\beta=\alpha^{1+q^2+q^4+q^6}\in \F_{q^2}\setminus \{0\}.$
From the first equation of System (\ref{sist sole x caso gcd 2 n=8}) we can solve for $X_0,$ and by replacing it in the second one, we obtain \begin{equation}
    aX_1^2+bX_1+c=0,
\end{equation} where 
$$ a=a_0+a_1X_2+a_2X_2^2, \quad 
     b=b_0+b_1X_2+b_2X_2^2,\quad
     c=c_0+c_1X_2+c_2X_2^2,
$$
 with in particular
 \begin{eqnarray*}
     a_2&=&\beta^q(X_3 - X_6)( \beta X_3 X_4 -  \beta X_3 X_7 -  \beta X_4 X_5 +  \beta X_5 X_7 + X_4 X_6 - X_4 X_7 - X_5 X_6 + X_5 X_7),\\
     b_2&=&  \beta^{q+1}X_3^2X_4^2 +  \beta^{q+1}X_3^2X_4X_6 -  \beta^{q+1}X_3^2X_4X_7 -  \beta^{q+1}X_3^2X_6X_7 -  \beta^{q+1}X_3X_4^2X_5\\&& -  \beta^{q+1}X_3X_4^2X_6 -
         \beta^{q+1}X_3X_4X_5X_6 +  \beta^{q+1}X_3X_4X_5X_7 -  \beta^{q+1}X_3X_4X_6^2\\&& +  \beta^{q+1}X_3X_4X_6X_7 +  \beta^{q+1}X_3X_5X_6X_7 +  \beta^{q+1}X_3X_6^2X_7 +
         \beta^{q+1}X_4^2X_5X_6 +  \beta^{q+1}X_4X_5X_6^2\\&& -  \beta^{q+1}X_4X_5X_6X_7 -  \beta^{q+1}X_5X_6^2X_7 + \beta X_3^2X_4X_5 -  \beta X_3^2X_4X_6 -
         \beta X_3^2X_5X_7 +  \beta X_3^2X_6X_7\\&& -  \beta X_3X_4^2X_5 +  \beta X_3X_4^2X_6 +  \beta X_3X_5X_7^2 -  \beta X_3X_6X_7^2 +  \beta X_4^2X_5X_7 -
         \beta X_4^2X_6X_7\\&& -  \beta X_4X_5X_7^2 +  \beta X_4X_6X_7^2 +  \beta^qX_3X_4^2X_6 -  \beta^qX_3X_4^2X_7 -  \beta^qX_3X_5^2X_6 +  \beta^qX_3X_5^2X_7\\&& -
         \beta^qX_4^2X_6^2 +  \beta^qX_4^2X_6X_7 +  \beta^qX_5^2X_6^2 -  \beta^qX_5^2X_6X_7 + X_3X_4X_5X_6 - X_3X_4X_5X_7\\&& - X_3X_4X_6^2 +
        X_3X_4X_6X_7 - X_3X_5^2X_6 + X_3X_5^2X_7 + X_3X_5X_6^2 - X_3X_5X_6X_7\\&& - X_4^2X_5X_6 + X_4^2X_5X_7 + X_4^2X_6^2 -
        X_4^2X_6X_7 + X_4X_5^2X_6 - X_4X_5^2X_7 - X_4X_5X_6^2\\&& + X_4X_5X_6X_7.
 \end{eqnarray*}

 \begin{prop} \label{prop a,b,c copr n=8 gcd 2}
The polynomials     $a,b,c$ have no non-trivial common factors.
\end{prop}
\begin{proof}
By contradiction let $d(X_2,X_3,X_4,X_5,X_6,X_7)=\gcd(a,b,c)$ of degree at least $1$. Arguing as in  the proof of Proposition \ref{prop_coprime}, we can prove that $d$ does not depend on $X_2$. Since $d\mid a,b$, we have that $d\mid a_2,b_2$. The factors of $a_2$ are both of degree $1$ in $X_3$ and  it is easy to see  that $\deg(\gcd(a_2,b_2))=0$, contradicting  $\deg(d)\geq 1$.
\end{proof} 

As in the previous Section  (see Remark \ref{rem irr se e solo se polo disp o delta no square}), to obtain the irreducibility of $\mathcal{Z}$, it is sufficient to provide the existence of a $6$-tuple $Q=(u_2,u_3,u_4,u_5,u_6,u_7)\in \overline{\mathbb{F}_q}^{6}$ which is a zero of odd multiplicity $m_1$ for the polynomial $ac$ and a zero of multiplicity $m_2$ for the polynomial $b$ such that $2m_2>m_1$. 
\begin{theorem} \label{thm scelta punto ad ok per n=8 e gcd 2}
    There exist a choice of $u_2,u_5,u_6\in \overline{\mathbb{F}_q}$ such that $Q=(u_2,\frac{\beta^qu_5u_6}{\beta^qu_5 + u_5 - u_6},0,u_5,u_6,0)$ is a simple zero of $c$, a zero of $b$, and not a zero of $a$. 
\end{theorem}
\begin{proof}

We can rewrite $a,b,c$ as
\begin{eqnarray*}
    a&=&(X_2-X_3)X_5X_6(\beta^qX_2X_3 - \beta^qX_2X_6 + X_3X_5 - X_3X_6)+X_4a_4(X_2,X_3,X_4,X_5,X_6)\\&&+X_7a_7(X_2,X_3,X_4,X_5,X_6,X_7),\\
    b&=&(X_2-X_3)X_2X_5X_6(\beta^qX_3X_5 - \beta^qX_5X_6 + X_3X_5 - X_3X_6)+X_4b_4(X_2,X_3,X_4,X_5,X_6)\\&&+X_7b_7(X_2,X_3,X_4,X_5,X_6,X_7),\\
    c&=&X_2X_4X_6c_4(X_2,X_3,X_4,X_5,X_6)+X_7c_7(X_2,X_3,X_4,X_5,X_6,X_7),
\end{eqnarray*}
with $X_4\nmid c_4$.
Note that $$a\left(X_2,\frac{\beta^qX_5X_6}{\beta^qX_5 + X_5 - u_6},0,X_5,X_6,0\right)=\frac{n_a(X_2,X_5,X_6)}{d_a(X_2,X_5,X_6)}$$ with $n_a\not\equiv0\not\equiv d_a$, and $$c_4\left(X_2,\frac{\beta^qX_5X_6}{\beta^qX_5 + X_5 - u_6},0,X_5,X_6\right)=\frac{n_c(X_2,X_5,X_6)}{d_c(X_2,X_5,X_6)}$$ with $n_c\not\equiv0\not\equiv d_c$.
This implies the existence of  $u_2,u_5,u_6\in \overline{\mathbb{F}_q}$ such that $$n_a(u_2,u_5,u_6)d_a(u_2,u_5,u_6)n_c(u_2,u_5,u_6)d_c(u_2,u_5,u_6)\neq0,$$ and the claim follows.
\end{proof}

\begin{cor} \label{cor Z irr gcd 2 n 8}
    Let $\beta=\alpha^{1+q^2+q^4+q^6}\in \F_{q^2}\setminus \{0\},$ and let $\mathcal{Z}$ be the variety defined by System \eqref{sist sole x caso gcd 2 n=8}. Then $\mathcal{Z}$ is absolutely irreducible.
\end{cor}
\begin{proof}
    It is a direct consequence of Remark \ref{rem irr se e solo se polo disp o delta no square}, Proposition \ref{prop a,b,c copr n=8 gcd 2}, and Theorem \ref{thm scelta punto ad ok per n=8 e gcd 2}.
\end{proof}

 \subsection{Case $q$ even, $n=8,$ $\gcd(n,J-I)=4$} \label{subsect q even n 8 gcd 4}
By Theorem \ref{thm equivalenza aggiunto} and Proposition \ref{prop equiv CMPZ somma eq n/2} it is enough to consider the case $I=1$ and $J=5.$\\
In order to provide non-existence results for scattered binomials of this shape, it is sufficient to establish the absolute irreducibility of the  variety
 \begin{equation}\label{sist sole x caso gcd 4}
    \mathcal{Z}:\begin{cases}
    (X_0-X_1)(X_4-X_5)-\beta(X_0-X_5)(X_4-X_1)=0 \\
    (X_1-X_2)(X_5-X_6)-\beta^q(X_1-X_6)(X_5-X_2)=0 \\
    (X_2-X_3)(X_6-X_7)-\beta^{q^2}(X_2-X_7)(X_6-X_3)=0 \\
    (X_3-X_4)(X_7-X_0)-\beta^{q^3}(X_3-X_0)(X_7-X_4)=0,
\end{cases}
\end{equation}
with $\beta=\alpha^{1+q^4} \in \F_{q^4}\setminus \{0,1\}$ (recall Remark \ref{rem norm no 1 and I < n/4}). From the first and second equations of System \eqref{sist sole x caso gcd 4} one  obtains 
\begin{equation*}
    X_0=\dfrac{(\beta+1)X_1X_5 + \beta X_4X_5 + X_1X_4 }{\beta X_1 + (\beta+1)X_4 + X_5}
\end{equation*}
and \begin{equation*}
    X_2=\dfrac{(\beta^q+1)X_1X_5 + \beta^qX_5X_6 + X_1X_6}{\beta^qX_1 + (\beta^q+1)X_6 + X_5}.
\end{equation*}
By replacing this expression of $X_2$ in the third equation of System \eqref{sist sole x caso gcd 4}, we obtain $R_1(X_1,X_3,X_5,X_6)\cdot X_7 + R_2(X_1,X_3,X_5,X_6)=0,$ where
\begin{eqnarray*}
    R_1
    &=&(\beta^{q+q^2}+\beta^q)X_1X_3 + (\beta^{q+q^2}+1)X_1X_6 + (\beta^{q+q^2}+\beta^q+\beta^{q^2}+1)X_3X_6 \\&&+ (\beta^{q+q^2}+\beta^{q^2})X_6^2  + (\beta^q+1)X_1X_5  + (\beta^q+\beta^{q^2 })X_5X_6 + (\beta^{q^2}+1)X_3X_5  \\
    R_2
    &=& (\beta^{q+q^2}+\beta^{q^2})X_1X_3X_5 + (\beta^{q+q^2}+\beta^q+\beta^{q^2}+1)X_1X_5X_6 + (\beta^{q+q^2}+1)X_3X_5X_6\\&& + (\beta^{q+q^2}+\beta^q)X_5X_6^2 + (\beta^q+\beta^{q^2})X_1X_3X_6  + (\beta^q+1)X_3X_6^2     + (\beta^{q^2}+1)X_1X_6^2.  
\end{eqnarray*}
Thus, after replacing such expressions of $X_0$ and $X_7$ in the fourth equation, the variety $\mathcal{Z}$ reads as 
\begin{equation}\label{sist Z gcd 4 prime tre razionali}
    \mathcal{Z}:\begin{cases}
X_0=\frac{(\beta+1)X_1X_5 + \beta X_4X_5 + X_1X_4 }{\beta X_1 + (\beta+1)X_4 + X_5} \\
    X_2=\frac{(\beta^q+1)X_1X_5 + \beta^qX_5X_6 + X_1X_6}{\beta^qX_1 + (\beta^q+1)X_6 + X_5} \\
    X_7=\frac{R_2(X_1,X_3,X_5,X_6)}{R_1(X_1,X_3,X_5,X_6)} \\
    S(X_1,X_3,X_4,X_5,X_6)=0,
\end{cases}
\end{equation} where $S(X_1,X_3,X_4,X_5,X_6)$ is a homogeneous polynomial of degree $5,$ and such that $\deg_{X_i}(G)=2,$ for $i=1,3,4,5,6.$ Therefore we can write \[S(X_1,X_3,X_4,X_5,X_6)=aX_1^2+bX_1+c,\] where $a,b,c \in \F_{q^3}[X_3,X_4,X_5,X_6]$ are homogeneous polynomials of degree $5,4,3$ respectively. 
\begin{remark} \label{rem irr Z iff irr S iff irr c}
    Since the first three equations of System \eqref{sist Z gcd 4 prime tre razionali} are rational in $X_0,X_2$ and $X_7, $ the variety defined by System \eqref{sist sole x caso gcd 4} is absolutely irreducible if and only if so is $S(X_1,X_3,X_4,X_5,X_6)$. Also, being $S$ homogeneous, a sufficient condition for $S(X_1,X_3,X_4,X_5,X_6)$ to be absolutely irreducible is that the polynomial $c=c(X_3,X_4,X_5,X_6)=S(0,X_3,X_4,X_5,X_6)$ is absolutely irreducible.
\end{remark} By direct computations, it follows that $c$ equals 
\begin{equation}
   \label{eq c caso gcd 4} (\beta+1)^{q^3}X_3X_5X_6\left[(\beta^{q+q^2}+1)X_3X_5 + (\beta^{q+q^2}+\beta^q)X_5X_6 + (\beta^q+1)X_3X_6\right]+T(X_3,X_4,X_5,X_6),
\end{equation}
where $\deg(T(X_3,X_4,X_5,X_6))=5$ and $X_4 \mid T(X_3,X_4,X_5,X_6).$  
\begin{prop} \label{prop irr c}
    Let $\beta \in \F_{q^3}\setminus \{0,1\}$ and $c$ as in Equation \eqref{eq c caso gcd 4}. Then $c$ is absolutely irreducible.
\end{prop}
\begin{proof}
    Suppose by contradiction that $c$ possesses a non-trivial factorization. Denote by $c_{*}(X_3,X_5,X_6)$ the dehomogenized of $c$ with respect to $X_4.$ Then also $c_{*}$ is reducible and from Equation \eqref{eq c caso gcd 4}, it follows that a factor  of $c_{*}$ of degree $\leq 2$ has to be equal to one of the following polynomials
    \begin{eqnarray*}
        X_3+u_1, \quad 
        X_5+u_2, \quad 
        X_6+u_3, \quad 
        X_3X_5+u_4X_3+u_5X_5+u_6X_6+u_7, \\ 
        X_3X_6+u_8X_3+u_9X_5+u_{10}X_6+u_{11}, \quad 
        X_5X_6+u_{12}X_3+u_{13}X_5+u_{14}X_6+u_{15},\\
        (\beta^{q+q^2}+1)X_3X_5 + (\beta^{q+q^2}+\beta^q)X_5X_6 + (\beta^q+1)X_3X_6+u_{16}X_3+u_{17}X_5+u_{18}X_6+u_{19},
    \end{eqnarray*}
    for suitable $u_1,\dots,u_{19}\in \overline{\fq}.$ Let consider each case separately. 
    \begin{itemize}
        \item $X_3+u_1$\\
        Let $\Tilde{c}=c_{*}(u_1,X_5,X_6).$ By direct computations  the coefficient of $X_5$ in $\Tilde{c}$ is $u_1^2\beta\qqq(\beta+1)^{1+q^2},$ that is zero if and only if $u_1=0.$ If $u_1=0,$ then $c_{*}(0,X_5,X_6)=\beta(\beta+1)^{q^3}(\beta+\beta\q)\q X_5^2X_6+\beta^q(\beta+1)\qq(\beta^{1+q^3}+1)X_5^2X_6^2+\gamma X_5X_6^2,$ with $\gamma \in \F_{q^3}.$ Since $\beta+\beta\q=0=\beta^{1+q^3}+1$ implies $\beta=1,$ it follows that $c_{*}(0,X_5,X_6)$ does not vanish.
        \item $X_5+u_2$\\
        Let $\Tilde{c}=c_{*}(X_3,u_2,X_6).$ By direct computations  the coefficients of $X_3$ and $X_3^2X_5^2$ in $\Tilde{c}$ are respectively $u_2^2\beta(\beta+1)^{q^2+q^3}$ and $(\beta+1)^{q+q^3}(\beta+u_2+1) .$  Such coefficients cannot be simultaneously zero, so $\Tilde{c}\neq 0.$
        \item $X_6+u_3$\\
        Let $\Tilde{c}=c_{*}(X_3,X_5,u_3).$ By direct computations  the coefficients of $X_3$ and $X_3^2$ in $\Tilde{c}$ are respectively $u_3^2(\beta+1)^{1+q}(\beta^{q^2+q^3}+1)$ and $u_3(\beta+1)^{1+q}(\beta^{q^2+q^3} + u_3 \beta\qqq + \beta\qqq + u_3) .$  Such coefficients are both zero if and only if $u_3=0$ or $\beta^{q^2+q^3}+1=0=\beta^{q^2+q^3} + u_3 \beta\qqq + \beta\qqq + u_3.$ The latter condition implies $(u_3+1)(\beta+1)^{q^3}=0,$ so $u_3=1.$ Now we get $c_{*}(X_3,X_5,0)=\beta(\beta+1)^{q^2+q^3}X_3X_5^2+\gamma_1X_3^2X_5+\theta_1X_3^2X_5^2$ and $c_{*}(X_3,X_5,1)=(\beta+1)^{q^2}(\beta^q+\beta)X_3^2X_5^2+L(X_3,X_5),$ with $\deg(L)\leq 3.$ It is straightforward that $c_{*}(X_3,X_5,0)\neq 0.$ Also, $\beta^q+\beta=0=\beta^{q^2+q^3}+1$ implies $\beta=1,$ we obtain $c_{*}(X_3,X_5,1)\neq 0.$ 
        \item $X_3X_5+u_4X_3+u_5X_5+u_6X_6+u_7$\\
        Let $\Tilde{c}(X_5,X_6)=(X_5+u_4)^2c_{*}(\frac{u_5X_5+u_6X_6+u_7}{X_5+u_4},X_5,X_6).$ Of course $X_3X_5+u_4X_3+u_5X_5+u_6X_6+u_7\mid c_{*}$ if and only if $\Tilde{c}(X_5,X_6)=0.$ By MAGMA computations,  the coefficients of $X_5,$  $X_6^4$, and $X_5^4X_6^2$ in $\Tilde{c}$ are respectively $u_7^2\beta\qqq(\beta+1)^{1+q^2},$  $u_6^2(\beta+1)^{1+q+q^3} $, and $\beta^q(\beta+1)^{q^2}(\beta^{1+q^3}+u_5\beta\qqq+u_5+1).$  They are all zero  if and only if $u_6=0=u_7$ and $u_5=\frac{\beta^{1+q^3}+1}{\beta\qqq+1}.$ By replacing such conditions in $\Tilde{c}(X_5,X_6),$ we obtain with the help of MAGMA that $\Tilde{c}(X_5,X_6)=\beta\qqq(\beta+1)^{2+q^2}(\beta^{1+q^3}+1)X_5^4+\beta\qqq (\beta+1)^{2+q^3}(\beta\q+\beta\qq)X_5^4X_6+L(X_5,X_6),$ where $\deg_{X_5}(L)\leq 3.$ Since $\beta\q+\beta\qq=0=\beta^{1+q^3}+1$ implies $\beta=1,$ we obtain $\Tilde{c}\neq 0.$
        \item $X_3X_6+u_8X_3+u_9X_5+u_{10}X_6+u_{11}$\\
        Let $\Tilde{c}(X_5,X_6)=(X_6+u_8)^2c_{*}(\frac{u_9X_5+u_{10}X_6+u_{11}}{X_6+u_8},X_5,X_6).$  By MAGMA computations,  the coefficients of $X_5,$  $X_5^3X_6^3$, and $X_5^2X_6^4$ in $\Tilde{c}$ are respectively $u_{11}^2\beta\qqq(\beta+1)^{1+q^2},$  $u_9\beta\q(\beta+1)^{q^2+q^3} $, and $\beta^q(\beta+1)^{q^2}(\beta^{1+q^3}+u_{10}\beta\qqq+u_{10}+1).$  They are all zero  if and only if $u_9=0=u_{11}$ and $u_{10}=\frac{\beta^{1+q^3}+1}{\beta\qqq+1}.$ By replacing such conditions in $\Tilde{c}(X_5,X_6),$  MAGMA computations show that the coefficient of $X_6^4$ in $\Tilde{c}(X_5,X_6)$ is $\beta\qqq(\beta+1)^{1+q+q^3}(\beta+\beta^{q^2})(\beta^{1+q^3}+1),$ that is zero if and only if $\beta^{1+q}=1$ or $\beta \in \fq.$ In the first case, the coefficient of   $X_5X_6^4$ in $\Tilde{c}(X_5,X_6)$ becomes $\beta(\beta+1)^{3+2q}.$ In the latter, it turns out to be $\beta^{1+q}(\beta+1)^{2+3q}.$ They are both non-zero, so $\Tilde{c}(X_5,X_6)\neq 0.$
        \item $X_5X_6+u_{12}X_3+u_{13}X_5+u_{14}X_6+u_{15}$ \\
        Let $\Tilde{c}(X_3,X_6)=(X_6+u_{13})^2c_{*}(X_3,\frac{u_{12}X_3+u_{14}X_6+u_{15}}{X_6+u_{13}},X_6).$  By MAGMA computations,  the coefficients of $X_3,$  $X_3^2X_6^4$ and $X_3^3X_6^3$ in $\Tilde{c}$ are respectively $u_{15}^2\beta(\beta+1)^{q^2+q^3},$  $(\beta+1)^{q+q^3}(\beta+u_{14}+1) $ and $u_{12}\beta^q(\beta+1)^{q+q^3}.$  They are all zero  if and only if $u_{12}=0=u_{15}$ and $u_{14}=\beta+1.$ By replacing such conditions in $\Tilde{c}(X_3,X_6),$ it follows that the coefficients of $X_3X_6^4$ and $X_3^2X_6^3$ in $\Tilde{c}(X_3,X_6)$ are respectively $\beta(\beta+1)^{1+q}(\beta\qq+\beta\qqq)$ and $(\beta+1)^{1+q}(\beta^{1+q^2}+\beta +u_{13}\beta\qqq+u_{13}).$ They are both zero if and only if $\beta \in \fq$ and $u_{13}=\frac{\beta^{1+q^2}+\beta}{\beta\qqq+1}.$ Combining these two conditions with the previous ones, one gets $\Tilde{c}(X_3,X_6)=\beta^2(\beta+1)^6X_6(X_6+1)(X_3+X_6)^2\neq 0.$
  \item $(\beta^{q+q^2}+1)X_3X_5 + (\beta^{q+q^2}+\beta^q)X_5X_6 + (\beta^q+1)X_3X_6+u_{16}X_3+u_{17}X_5+u_{18}X_6+u_{19}$ \\
  Let \\$\Tilde{c}(X_3,X_5)=\left[(\beta^{q+q^2}+\beta^q)X_5+(\beta\q+1)X_3+u_{18}  \right]^2c_{*}\left(X_3,X_5,\frac{(\beta^{q+q^2}+1)X_3X_5+u_{16}X_3+u_{17}X_5+u_{19}}{(\beta^{q+q^2}+\beta^q)X_5+(\beta\q+1)X_3+u_{18}}\right).$
  We want to prove that $\Tilde{c}(X_3,X_5)$ is not zero. By computations with MAGMA, we observe that the coefficient of $X_3$ in $\Tilde{c}(X_3,X_5)$ is $u_{19}^2(\beta+1)^{1+q}(\beta^{q^2+q^3}+1).$ Such coefficient is zero if and only if $u_{19}=0$ or $\beta^{1+q}=1.$ Consider these two cases separately. \\ \textbf{Case 1.} $\beta^{1+q}+1=0$\\
  Since $\beta^q=\beta^{-1},$ it follows by direct computations that the coefficient of $X_3^2X_5^4$ in $\Tilde{c}$ is $(\beta+1)^5,$ so $\Tilde{c}(X_3,X_5)\neq 0.$ \\
  \textbf{Case 2.}
  $u_{19}=0$\\
  We can suppose without restriction that $\beta^{1+q}+1\neq 0.$
  By replacing $u_{19}=0$ in $\Tilde{c}(X_3,X_5),$ MAGMA computations show that the coefficient of $X_3^3X_5^3$ in $\Tilde{c}$ is $$(\beta+1)^{q^3}(\beta^{q+q^2}+1)(\beta^{q+q^2}u_{16} + \beta^{q+q^2}u_{18} + \beta\q u_{16} + \beta\q u_{17} + u_{17} + u_{18}).$$  Such a coefficient vanishes if and only if $u_{16}=\frac{(\beta^q+1)u_{17}+(\beta^{q+q^2}+1)u_{18}}{\beta^q(\beta+1)^{q^2}}.$ This implies that the coefficient of $X_3^2X_5^4$ in $\Tilde{c}(X_3,X_5)$ is $$\beta^{3q}(\beta+1)^{3q^2}[(\beta+1)\qqq(\beta^{q+q^2}+1)u_{17}+\beta\qq(\beta+1)^{2q}(\beta+\beta\qqq)],$$ that is zero if and only if $u_{17}=\frac{\beta\qq(\beta+1)^{2q}(\beta+\beta\qqq)}{(\beta+1)\qqq(\beta^{q+q^2}+1)}.$ By replacing such an expression of $u_{17}$ in $\Tilde{c}(X_3,X_5),$ the coefficient of $X_3^4X_5^2$ turns out to be $$\beta^q(\beta+1)^{1+q^2+2q^3}(\beta^{q+q^2}+1)(\beta^{1+q+q^2} + \beta^{1+q^2} + \beta^{q+q^2+q^3} + \beta^{q^2+q^3} + \beta\qqq u_{18} + u_{18}),$$ which is zero if and only if $u_{18}=\frac{\beta\qq(\beta+1)\q(\beta+\beta\qqq)}{(\beta+1)\qqq}.$ By substituting such an expression of $u_{18}$ in $\Tilde{c}(X_3,X_5),$ we get that  the coefficient of $X_3^3$ in $\Tilde{c}(X_3,X_5)$ is $$\beta^{3q+2q^2}(\beta+1)^{1+3q+4q^2+2q^3}(\beta\q+\beta\qqq)(\beta+\beta\qqq),$$ which is not zero unless $\beta \in \F_{q^2}.$ Finally, if $\beta \in \F_{q^2},$ we observe by direct computations that the coefficient of $X_3X_5^4$ in $\Tilde{c}(X_3,X_5)$ is $\beta^{1+4q}(\beta+1)^{7+7q}\neq 0.$ It follows that $\Tilde{c}(X_3,X_5)\neq 0.$
    \end{itemize} 
The statement follows.
\end{proof}
The following result is a direct consequence of Remark \ref{rem irr Z iff irr S iff irr c} and Proposition \ref{prop irr c}.
\begin{prop} \label{prop Z irr gcd 4 n 8 q even}
    Let $q=2^n,$  $\beta=\alpha^{1+q^4} \in \F_{q^4}\setminus \{0,1\},$ and $\mathcal{Z}$ as in  \eqref{sist sole x caso gcd 4}. Then $\mathcal{Z}$ is an absolutely irreducible variety.
\end{prop}
\subsection{Main result for scattered binomials when $n \leq 8$}
Thanks to the previous results in this section and the machinery involving the algebraic varieties introduced in Section \ref{sec link variety scatt}, we are able to prove the following asymptotic classification result for scattered binomials.
\begin{theorem}
    Let $3\leq n \leq 8,$  $q\geq 162019556021,$ and $0<I<J<n.$ Then $f(x)=x^{q^I}+\alpha x^{q^J} $ is scattered on $\Fn$ if and only if one of the following conditions holds:\begin{itemize}
        \item [(i)]$\gcd(I,n)=1,$  $I+J=n,$ and $\textnormal{N}_{\Fn / \fq}(\alpha)\neq 1;$
        \item [(ii)] $n=6,$ $J-I=3,$ and $\textnormal{N}_{q^6 / q^3}(\alpha)=\alpha^{1+q^3}$ satisfies the  conditions of Theorem \ref{thm n=6 caratt scatt};
        \item [(iii)]$q$  odd, $n=8,$ $\gcd(I,n)=1,$ $J-I=4$, and $\textnormal{N}_{q^8 / q^4}(\alpha)=\alpha^{1+q^4}=-1.$ 
    \end{itemize}  
\end{theorem}
\begin{proof}
$(\Leftarrow) \ $ It follows from \cite{LunardonPolverino2001,sheekey2016new}, Theorems \ref{thm n=6 caratt scatt}, \ref{thm n=8 timp zini}, 
 \ref{thm equivalenza aggiunto}, and Proposition \ref{prop equiv CMPZ somma eq n/2}. \\
 $(\Rightarrow) \ $ Assume that $f$ does not satisfy any of the conditions $(i),(ii),(iii).$ By \cite[Theorem 3.4]{zanella2019condition} and  Theorems \ref{thm n=6 caratt scatt} and \ref{thm n=8 timp zini} we can assume without restriction $I,J,n$ as at the beginning of Section \ref{sect
class bin n fino a 8}. Consider the varieties $\mathcal{C}^{\prime}_{f},\mathcal{V},\mathcal{W},\mathcal{Z}$ introduced in Subsection \ref{sec link variety scatt} and in Section \ref{sec struc f W}.   By Corollaries \ref{cor Z irr gcd 2 n 6}, \ref{cor Z irr gcd 2 n 8} and Proposition \ref{prop Z irr gcd 4 n 8 q even} it follows that $\mathcal{Z}$ is an absolutely irreducible variety. Then Proposition \ref{prop W meno x y ass irr iff Z ass irr} implies that $\mathcal{W} \setminus (\mathcal{X} \cup \mathcal{Y})$ is absolutely irreducible, and it corresponds via the projectivity \[(x_0,\dots,x_{n-1},y_0,\dots,y_{n-1}) \mapsto (X_0,\dots,X_{n-1},Y_0,\dots,Y_{n-1})\] of $\mathbb{P}^{2n-1}(\Fn)$ defined in Subsection \ref{sec link variety scatt} to an absolutely irreducible component of $\mathcal{V}$ defined over $\fq,$ not contained in $x_0=x_1=\dots=x_{n-1}$ or $y_0=y_1=\dots=y_{n-1}=0.$ The statement follows as in the proof of Theorem \ref{main theorem}.
\end{proof}
The previous result addresses in particular a conjecture stated in \cite{TimpanellaZini2024} (see also \cite[Remark 7.4]{csajbok2018new}).
\begin{cor}
    There are no scattered binomial over $\F_{q^8}$ of type $f(x)=x^{q^s}+\delta x^{q^s+4},$ for $q$ even and large enough.
\end{cor}
\section{The variety $\mathcal{Z}$ for the known scattered binomials}
We have already observed in Remark \ref{rem produttore uguali iff I+J=n} that for $I+J=n$ and $\textnormal{N}_{q^n/q}(\alpha)\neq 1,$ the variety $\mathcal{W}\setminus (\mathcal{X}\cup\mathcal{Y})$ splits in $n$ components. 
In this concluding section, we aim to analyze the behavior of the variety $\mathcal{Z}$, and consequently that of $\mathcal{W} \setminus (\mathcal{X} \cup \mathcal{Y})$, when the binomial belongs to the families of scattered binomials introduced in \cite{csajbok2018new} for $n=6,8$.

 In particular we will show that $\mathcal{Z}$ decomposes into absolutely irreducible components not fixed by the collineation $\Psi_q,$ introduced in Subsection \ref{sec link variety scatt}, which correspond to absolutely irreducible components of $\mathcal{V}$ not defined over $\fq.$ Finally, starting from this decomposition, we will derive the  conditions on the coefficient $\alpha$ obtained in \cite{csajbok2018new}. We will consider the cases $n=6$ and $n=8$ separately.
\subsection{Case $n=6$}
In the following we will consider binomials of type $f(x)=x^q+\alpha x^{q^4} \in \F_{q^6}[x],$ with $\textnormal{N}_{q^6 / q^3}(\alpha)=\alpha^{1+q^3}\neq 1.$ Such a family has already been studied in \cite{BartoliCsajbokMontanucci2021,PolverinoZullo2020rootsof...}, where the authors proved, using different techniques, necessary and sufficient conditions on the coefficient $\alpha \in \F_{q^6}$ to obtain scattered  binomials. In what follows we will show how to get the same conditions on $\alpha$ by analyzing the decomposition of $\mathcal{W}$ in absolutely irreducible components. Set $\beta=\alpha^{1+q^3}\in \F_{q^3}.$\\
By direct computations, it is possible to check that the variety \begin{equation}\label{eq L1L2L3}
    \mathcal{Z}:\begin{cases}
    L_1:=(X_0-X_1)(X_3-X_4)-\beta(X_0-X_4)(X_3-X_1)=0 \\
    L_2:=(X_1-X_2)(X_4-X_5)-\beta^q(X_1-X_5)(X4-X_2)=0 \\
    L_3:=(X_2-X_3)(X_5-X_0)-\beta^{q^2}(X_2-X_0)(X_5-X_3)=0 
\end{cases}
\end{equation}  reads
\begin{equation}\label{sist sole x caso gcd 3 n eq 6}
    \mathcal{Z}:\begin{cases}
    X_0=\frac{(\beta-1)X_1X_4 - \beta X_3X_4 + X_1X_3 }{\beta X_1 + (1-\beta)X_3 - X_4} \\
    X_2=\frac{(\beta^q-1)X_1X_4 - \beta^qX_4X_5 + X_1X_5}{\beta^qX_1 + (1-\beta^q)X_5 - X_4} \\
    G(X_1,X_3,X_4,X_5)=0, 
\end{cases}
\end{equation} with  
    \begin{eqnarray*}
        G(X_1,X_3,X_4,X_5)&=&(\beta -1)^{1+q} \cdot F_1(X_{1},X_{3},X_{4},X_{5})\cdot F_2(X_{1},X_{3},X_{4},X_{5})\\&=&(\beta -1)^{1+q} \cdot  (X_{1}X_{4}+AX_{1}X_{3}+BX_{1}X_{5}+BX_{3}X_{4}+AX_{4}X_{5}+X_{3}X_{5})\\ &&\cdot(X_{1}X_{4}+CX_{1}X_{3}+DX_{1}X_{5}+DX_{3}X_{4}+CX_{4}X_{5}+X_{3}X_{5}),
    \end{eqnarray*}
   where  $A,C \in  \F_{q^6}$ and $B,D \in  \F_{q^6}$ are respectively the roots of \begin{equation}\label{eq p1 cond A,C}
        p_1(T):=(\beta-1)^{1+q}T^2+(\beta^{1+q} - \beta^{1+q^2} + \beta^{q+q^2} - 2\beta^q + 1)T+\beta^q(\beta-1)^{q^2}
    \end{equation} and 
    \begin{equation}\label{eq p2 cond B,D}
        p_2(T):=(\beta-1)^{1+q}T^2+(\beta^{1+q} + \beta^{1+q^2} - \beta^{q+q^2} - 2\beta + 1)T+\beta(\beta-1)^{q^2}.
    \end{equation}
The following proposition is now immediate.
\begin{cor}
    Let $\mathcal{Z}$ be defined by System \eqref{sist sole x caso gcd 3 n eq 6}. Then $\mathcal{Z}$ decomposes in two absolutely irreducible components $\mathcal{Z}=\mathcal{W}_{1} \cup \mathcal{W}_{2},$ where \begin{equation}
   \mathcal{W}_i:\begin{cases}
    X_0=\frac{(\beta-1)X_1X_4 - \beta X_3X_4 + X_1X_3 }{\beta X_1 + (1-\beta)X_3 - X_4} \\
    X_2=\frac{(\beta^q-1)X_1X_4 - \beta^qX_4X_5 + X_1X_5}{\beta^qX_1 + (1-\beta^q)X_5 - X_4} \\
    F_i(X_1,X_3,X_4,X_5)=0, 
\end{cases} \ \ \textnormal{for} \  i=1,2.
\end{equation}
\end{cor}
\begin{remark} \label{rem psi w1 eq w2 o psi w1 eq w1}
    Since $\Psi_q(\mathcal{Z})=\mathcal{Z},$ it follows that either $\Psi_q(\mathcal{W}_1)=\mathcal{W}_2$ and $\Psi_q(\mathcal{W}_2)=\mathcal{W}_1,$ or $\Psi_q(\mathcal{W}_1)=\mathcal{W}_1$ and $\Psi_q(\mathcal{W}_2)=\mathcal{W}_2.$ In the latter case the components $\mathcal{W}_1$ and $\mathcal{W}_2$ correspond via the projectivity over $\Fn$ defined in Subsection \ref{sec link variety scatt} to absolutely irreducible components of $\mathcal{V}$ defined over $\fq.$ Then Theorem \ref{thm lang weil versione tredici terzi}  and Proposition \ref{link with curves} yield that $f(x)=x^{q}+\alpha x^{q^4}$ is not scattered over $\F_{q^6},$ for $q$ large enough. \\ In the former case $\Psi_q^2(\mathcal{W}_1)=\mathcal{W}_1,$  $\Psi_q^2(\mathcal{W}_2)=\mathcal{W}_2$, and $\mathcal{W}_1,$ $\mathcal{W}_2$ correspond to absolutely irreducible components of $\mathcal{V}$ defined over $\F_{q^2}.$ Thus $f(x)=x^{q}+\alpha x^{q^4}\in \F_{q^6}[x]$ is scattered if and only if there are no points fixed by the collineation $\Psi_q$ in $\mathcal{W}_1\cap \mathcal{W}_2.$  One can check that there are no such points in $\mathcal{Z},$ and thus  $f$ is scattered on $\F_{q^6}.$
\end{remark}

\begin{prop} \label{prop una il frob dell'altra se A in Fq3}
We have that $\Psi_q(\mathcal{W}_1)=\mathcal{W}_2$ and $\Psi_q(\mathcal{W}_2)=\mathcal{W}_1$ if and only if $A,B,C,D \in \F_{q^3}.$    
\end{prop}
\begin{proof}[Sketch of the proof]
 Since $\mathcal{Z}$ is fixed by $\Psi_q,$ it follows that $\Psi_q(\mathcal{W}_1)=\mathcal{W}_2$ (resp. $\Psi_q(\mathcal{W}_1)=\mathcal{W}_1$) if and only if $\Psi_q^3(\mathcal{W}_1)=\mathcal{W}_2$ (resp. $\Psi_q^3(\mathcal{W}_1)=\mathcal{W}_1$). \\ Moreover, from Equation \eqref{eq L1L2L3} we observe that $\Psi_q^3(L_i)=L_i,$ for $i=1,2,3,$ so $$\Psi_q^3(\mathcal{W}_1) = \begin{cases}
 L_1=0\\
 L_2=0\\
 \Psi_q^3(F_1)=0,
 \end{cases}$$ where \[\Psi_q^3(F_1)=X_1X_4+A\qqq(X_0X_4+X_1X_2)+B\qqq(X_2X_4+X_0X_1)+X_0X_2.\]
 If $A,B \in \F_{q^3}$ (so in particular $C,D \in \F_{q^3}$), then one can check by direct computations that $\Psi_q^3(F_1)$ lies in the ideal generated by $L_1,L_2,F_2$ in $\F_{q^6}[X_0,\dots,X_5],$ i.e. $\Psi_q^3(\mathcal{W}_1)=\mathcal{W}_2.$ \\If $A,B \in \F_{q^6} \setminus \F_{q^3},$ then by Equations \eqref{eq p1 cond A,C} and \eqref{eq p2 cond B,D} we get $A\qqq=C$ and $B\qqq=D$.  It is not difficult to see that $\Psi_q^3(F_1) \in \langle L_1,L_2,F_1 \rangle \subseteq \F_{q^6}[X_0,\dots,X_5]$ and $\Psi_q^3(\mathcal{W}_1)=\mathcal{W}_1.$
    \end{proof}

    \begin{remark}
    Remark \ref{rem psi w1 eq w2 o psi w1 eq w1} and 
     Proposition \ref{prop una il frob dell'altra se A in Fq3} show that $f=x^q+\alpha x^{q^4} \in \F_{q^6}[x]$ is scattered if and only if $A,B,C,D \in \F_{q^3}.$ This is consistent with \cite[Theorem 1.1]{BartoliCsajbokMontanucci2021} and Theorem \ref{thm n=6 caratt scatt}, since Equation \eqref{eq p1 cond A,C} (or equiv. Equation \eqref{eq p2 cond B,D}) admits both roots in $\F_{q^3}$ if and only if Equation \eqref{eq DELTA caso n=6 gcd 3} admits both roots in $\fq.$ To show this, we distinguish the cases $q$  odd and $q$ even. 
     
     For $q$ odd, we deduce   from Equations \eqref{eq p1 cond A,C} and \eqref{eq p2 cond B,D}  that  the discriminants of $p_1$ and $p_2$ with respect to $Z$ coincide and they are equal to  \begin{equation}
        \label{eq DELTA caso n=6 gcd 3} \Delta:=\textnormal{Tr}_{q^3 / q}(\beta^{2+2q})-2(\textnormal{N}_{q^3 / q}(\beta)\textnormal{Tr}_{q^3 / q}(\beta)+\textnormal{Tr}_{q^3 / q}(\beta^{1+q}))+8\textnormal{N}_{q^3 / q}(\beta)+1.
    \end{equation}
     Equation \eqref{eq DELTA caso n=6 gcd 3} yields $\Delta \in \fq,$ so $A,B,C,D \in \F_{q^3}$ if and only if $\Delta$ is a square in $\fq.$  One can check that $\Delta=\Big(\textnormal{N}_{q^3 / q}(\beta-1)^2\Big)\Big(\textnormal{Tr}_{q^3 / q}(\frac{\beta}{\beta-1})-1)^2-4\textnormal{N}_{q^3 / q}(\frac{\beta}{\beta-1})\Big),$ where the second factor  concides with the discriminant of Equations \eqref{eq tr N n=6 caratt scatt} in Theorem \ref{thm n=6 caratt scatt}. Thus the discriminant of Equation \eqref{eq p1 cond A,C} is a square in $\fq$ if and only if the discriminant of Equation \eqref{eq tr N n=6 caratt scatt} is a square in $\fq.$ 
     
     For $q$ even, Equations \eqref{eq p1 cond A,C} and  \eqref{eq p2 cond B,D} admit both roots in $\F_{q^3}$ if and only if $$\textnormal{Tr}_{q^3 / 2}\left( \dfrac{\beta \cdot \textnormal{N}_{q^3 / q}(\beta-1)}{\textnormal{N}_{q^3 / q}(\beta^{2+2q})+1}\right)=\textnormal{Tr}_{q / 2}\left( \dfrac{\textnormal{Tr}_{q^3 / q}(\beta) \cdot \textnormal{N}_{q^3 / q}(\beta-1)}{\textnormal{N}_{q^3 / q}(\beta^{2+2q})+1}\right)=0$$ while Equation \eqref{eq tr N n=6 caratt scatt} has both roots in $\fq$ if and only if $$\textnormal{Tr}_{q / 2}\left( \dfrac{ \textnormal{N}_{q^3 / q}(\frac{\beta}{\beta-1})}{(\textnormal{Tr}_{q^3 / q}(\frac{\beta}{\beta-1})-1)^2}\right)=\textnormal{Tr}_{q / 2}\left( \dfrac{\textnormal{N}_{q^3 / q}(\beta) \textnormal{N}_{q^3 / q}(\beta-1)}{\textnormal{N}_{q^3 / q}(\beta^{2+2q})+1}\right)=
     0.$$ One can check by direct computations that$$ \textnormal{Tr}_{q / 2}\left( \dfrac{\textnormal{Tr}_{q^3 / q}(\beta) \cdot \textnormal{N}_{q^3 / q}(\beta-1)}{\textnormal{N}_{q^3 / q}(\beta^{2+2q})+1}\right) + \textnormal{Tr}_{q / 2}\left( \dfrac{\textnormal{N}_{q^3 / q}(\beta) \textnormal{N}_{q^3 / q}(\beta-1)}{\textnormal{N}_{q^3 / q}(\beta^{2+2q})+1}\right)=0,$$ so $A,B,C,D\in \F_{q^3}$  if and only if  Equation \eqref{eq tr N n=6 caratt scatt} has two roots in $\fq.$
    \end{remark}
    \subsection{Case $n=8$}
    Set $f(x)=x^{q^I}+\alpha x^{q^{I+4}}\in \F_{q^8}[x] $ and $\beta=\alpha^{1+q^4}.$ Consider the variety $\mathcal{Z}$ defined  in System \eqref{sist sole x caso gcd 4}. By mimicking the  computations of Subsection \ref{subsect q even n 8 gcd 4}, we obtain 
    \begin{equation}
   \mathcal{Z}:\begin{cases}
    X_0=\frac{(\beta-1)X_1X_5 - \beta X_4X_5 + X_1X_4 }{\beta X_1 + (1-\beta)X_4 - X_5} \\
    a_1(X_1,X_3,X_5,X_7)X_{4}^2 +b_1(X_1,X_3,X_5,X_7)X_{4}+ c_1(X_1,X_3,X_5,X_7)=0\\
    X_2=\frac{(\beta^q-1)X_1X_5 - \beta^qX_5X_6 + X_1X_6}{\beta^qX_1 + (1-\beta^q)X_6 - X_5} \\
    a_2(X_1,X_3,X_5,X_7)X_{6}^2 +b_2(X_1,X_3,X_5,X_7)X_{6}+ c_2(X_1,X_3,X_5,X_7)=0, 
\end{cases} 
\end{equation}
 where \begin{eqnarray} \label{eq ai}
    a_1&=&\beta(\beta\qqq-1)X_{5} + \beta\qqq(1-\beta)X_{3}  + (\beta-1)X_{7} + (1-\beta\qqq)X_{1}, \\  \label{eq bi}
    b_1&=& (\beta-1)(\beta\qqq-1)(-X_{1}X_{5}+X_{3}X_{7}) +(\beta^{1+q^3} -1)(X_{1}X_{3}  -X_{5}X_{7}),\\ \nonumber &&+(\beta\qqq-\beta)(X_{1}X_{7} -X_{3}X_{5})\\   \label{eq ci}
    c_1&=& (\beta-1)(\beta\qqq X_{1}X_{5}X_{7}-X_{1}X_{3}X_{5}) + (\beta\qqq-1)(X_{3}X_{5}X_{7}-\beta X_{1}X_{3}X_{7}) , 
\end{eqnarray}
and $a_2=\Psi_q^2(a_1),$ $b_2=\Psi_q^2(a_1),$ $c_2=\Psi_q^2(c_1).$
Now set $\Delta_1:=b_1^2-4a_1c_1$ and $\Delta_2=a_2^2-b_2c_2.$ Then $\Psi_q^2(\Delta_1)=\Delta_2$ and one can check that \begin{eqnarray*}
        \Delta_1&=& (\beta-1)^{2+2q^3}  F_1(X_{1},X_{3},X_{5},X_{7})F_2(X_{1},X_{3},X_{5},X_{7})\\&=&(\beta-1)^{2+2q^3} (X_{1}X_{5}+AX_{1}X_7+BX_{1}X_{3}+BX_{5}X_{7}+AX_{3}X_{5}+X_{3}X_{7})\\ &&\cdot(X_{1}X_{5}+CX_{1}X_7+DX_{1}X_{3}+DX_{5}X_{7}+CX_{3}X_{5}+X_{3}X_{7}),
    \end{eqnarray*}
    where  $A,C\in \F_{q^4}$ (resp. $B,D \in \F_{q^4}$), are the roots of \begin{eqnarray*}
        p_1(T)&:=& (\beta-1)^{2+2q^3} T^2-2(\beta+\beta\qqq)(\beta-1)^{1+q^3}T+(\beta\qqq-\beta)^2\\
       \big(\textnormal{resp. of } \ p_2(T)&:=& (\beta-1)^{2+2q^3}  T^2+2(\beta^{1+q^3}+1)(\beta-1)^{1+q^3}T+(\beta^{1+q^3}-1)^2\big).
    \end{eqnarray*} 
    It is immediate to verify that $A\neq C$ and $B\neq D,$  so $F_1$ and $F_2$ are distinct absolutely irreducible polynomials and $\Delta_1,\Delta_2$ are not squares in $\F_{q^4}(X_1,X_3,X_5,X_7).$ This implies that $a_1X_4^2 + b_1X_4 + c_1$ and $a_2X_6^2 + b_2X_6 + c_2 = \Psi_q^2(a_1X_4^2 + b_1X_4 + c_1)$ are also absolutely irreducible.

    Therefore a necessary and sufficient condition for $\mathcal{Z}$ to be reducible is that $\Delta_1 = \gamma \Delta_2$, $\gamma \in \overline{\mathbb{F}_q}$,  i.e. either $(A,C)=(B\qq,D\qq)$ or   $(A,C)=(D\qq,B\qq)$.  
    
    This is equivalent to saying that the polynomials $p_1(T)$ and $p_2(T)\qq$ have the same roots, that is 
    \begin{equation}\label{Eq:finale1}
        - \dfrac{\beta+\beta\qqq}{(\beta-1)^{1+q^3}}=\dfrac{\beta^{q+q^2}+1}{(\beta-1)^{q+q^2}} \ \ \ \ \ \textnormal{and} \ \ \ \ \ \ \dfrac{(\beta\qqq-\beta)^2}{(\beta-1)^{2+2q^3}}=\dfrac{(\beta^{q+q^2}-1)^2}{(\beta-1)^{2q+2q^2}}.
    \end{equation}

    Note that $\beta=-1$ trivially verifies the above conditions. On the other hand, \eqref{Eq:finale1} implies $(\beta\qqq-\beta)(\beta^{q+q^2}+1)=\pm(\beta\qqq+\beta)(\beta^{q+q^2}-1),$ i.e., $\beta^{q^3+q^2+q-1}=1.$ Since $\beta \in \F_{q^4}\setminus\{1\}$ and $\gcd(q^4-1,q^3+q^2+q-1)=\gcd(2q-2,q^3+q^2+q-1)=2,$ one finally deduces $\beta=\alpha^{1+q^4}=-1.$

\section*{Acknowledgements}
The authors thank the Italian National Group for Algebraic and Geometric Structures and their Applications (GNSAGA—INdAM)
which supported the research. 

\section*{Declarations}
{\bf Conflicts of interest.} The authors have no conflicts of interest to declare that are relevant to the content of this
article.

\bibliographystyle{acm}

\begin{thebibliography}{10}

\bibitem{Bartoli:2020aa4}
{\sc Bartoli, D.}
\newblock Hasse-{W}eil type theorems and relevant classes of polynomial
  functions.
\newblock {\em London Mathematical Society Lecture Note Series, Proceedings of
  28th British Combinatorial Conference, Cambridge University Press\/} (2021),
  43--102.

\bibitem{BartoliCsajbokMontanucci2021}
{\sc Bartoli, D., Csajb\'ok, B., and Montanucci, M.}
\newblock On a conjecture about maximum scattered subspaces of {$\mathbb{F}_{q^6}
  \times \mathbb{F}_{q^ 6}$}.
\newblock {\em Linear Algebra Appl. 631\/} (2021), 111--135.

\bibitem{BartoliLongobardiMarinoTimpanella2024}
{\sc Bartoli, D., Longobardi, G., Marino, G., and Timpanella, M.}
\newblock Scattered trinomials of {$\mathbb F_{q^6}[X]$} in even characteristic.
\newblock {\em Finite Fields Appl. 97\/} (2024), Paper No. 102449, 28.

\bibitem{MR4173668}
{\sc Bartoli, D., Zanella, C., and Zullo, F.}
\newblock A new family of maximum scattered linear sets in {${\rm PG}(1,
  q^6)$}.
\newblock {\em Ars Math. Contemp. 19}, 1 (2020), 125--145.

\bibitem{bartoli2018exceptional}
{\sc Bartoli, D., and Zhou, Y.}
\newblock Exceptional scattered polynomials.
\newblock {\em J. Algebra 509\/} (2018), 507--534.

\bibitem{blokhuis2000scattered}
{\sc Blokhuis, A., and Lavrauw, M.}
\newblock Scattered spaces with respect to a spread in {PG}$(n, q)$.
\newblock {\em Geom. Dedicata 81}, 1 (2000), 231--243.

\bibitem{MR2206396}
{\sc Cafure, A., and Matera, G.}
\newblock Improved explicit estimates on the number of solutions of equations
  over a finite field.
\newblock {\em Finite Fields Appl. 12}, 2 (2006), 155--185.



\bibitem{Csajbokmarinopolverino2018classesofequivalence}
{\sc Csajb{\'o}k, B., Marino, G. and Polverino, O.}
\newblock Classes and equivalence of linear sets in {${\rm PG}(1,
              q^n)$}.
\newblock {\em J. Combin. Theory Ser. A 157\/} (2018), 402--426.

\bibitem{csajbok2018new}
{\sc Csajb{\'o}k, B., Marino, G., Polverino, O., and Zanella, C.}
\newblock A new family of {MRD}-codes.
\newblock {\em Linear Algebra Appl. 548\/} (2018), 203--220.

\bibitem{csajbok2018new2}
{\sc Csajb{\'o}k, B., Marino, G., and Zullo, F.}
\newblock New maximum scattered linear sets of the projective line.
\newblock {\em Finite Fields Appl. 54\/} (2018), 133--150.

\bibitem{csajbok2016pseudoregulus}
{\sc Csajb\'ok, B., and Zanella, C.}
\newblock On scattered linear sets of pseudoregulus type in
  {$\text{PG}(1,q^t)$}.
\newblock {\em Finite Fields Appl. 41\/} (2016), 34--54.

\bibitem{csajbokzanella2018PG1q4}
{\sc Csajb\'ok, B., and Zanella, C.}
\newblock Maximum scattered {$\mathbb F_q$}-linear sets of {${\rm PG}(1,q^4)$}.
\newblock {\em Discrete Math. 341}, 1 (2018), 74--80.

\bibitem{HKT}
  {\sc Hirschfeld, J.~W.~P., Korchm{\'a}ros, G., and Torres, F.},
  \newblock{\em Algebraic curves over a finite field}, vol.~20.
  \newblock Princeton University Press, 2013


\bibitem{LavMarPolTro2015}
{\sc Lavrauw, M., Marino, G., Polverino, O., and Trombetti, R.}
\newblock Solution to an isotopism question concerning rank 2 semifields.
\newblock {\em J. Combin. Des. 23}, 2 (2015), 60--77.

\bibitem{longobardi2023large}
{\sc Longobardi, G., Marino, G., Trombetti, R., and Zhou, Y.}
\newblock A large family of maximum scattered linear sets of {$PG(1, q^n)$} and
  their associated mrd codes.
\newblock {\em Combinatorica 43}, 4 (2023), 681--716.

\bibitem{longobardi2021linear}
{\sc Longobardi, G., and Zanella, C.}
\newblock Linear sets and {MRD}-codes arising from a class of scattered
  linearized polynomials.
\newblock {\em J. Algebraic Combin.\/} (2021), 1--23.

\bibitem{LunardonPolverino2001}
{\sc Lunardon, G., and Polverino, O.}
\newblock Blocking sets and derivable partial spreads.
\newblock {\em J. Algebraic Combin. 14}, 1 (2001), 49--56.

\bibitem{lunardon2018generalized}
{\sc Lunardon, G., Trombetti, R., and Zhou, Y.}
\newblock Generalized twisted {G}abidulin codes.
\newblock {\em J. Combin. Theory Ser. A 159\/} (2018), 79--106.

\bibitem{marino2020mrd}
{\sc Marino, G., Montanucci, M., and Zullo, F.}
\newblock {MRD}-codes arising from the trinomial $x^q+ x^{q^3}+ cx^{q^5}\in
  \mathbb{F}_{q^6} [x]$.
\newblock {\em Linear Algebra Appl. 591\/} (2020), 99--114.

\bibitem{NPZ}
{\sc Neri, A., Santonastaso, P., and Zullo, F.}
\newblock Extending two families of maximum rank distance codes.
\newblock {\em Finite Fields Appl. 81\/} (2022), 102045.

\bibitem{PolverinoZullo2020rootsof...}
{\sc Polverino, O., and Zullo, F.}
\newblock On the number of roots of some linearized polynomials.
\newblock {\em Linear Algebra Appl. 601\/} (2020), 189--218.

\bibitem{shafarevich1994basic}
{\sc Shafarevich, I.~R., and Reid, M.}
\newblock {\em Basic algebraic geometry}, vol.~2.
\newblock Springer, 1994.

\bibitem{sheekey2016new}
{\sc Sheekey, J.}
\newblock A new family of linear maximum rank distance codes.
\newblock {\em Adv. Math. Commun. 10}, 3 (2016), 475.

\bibitem{smaldore2024newscatteredlinearizedquadrinomials}
{\sc Smaldore, V., Zanella, C., and Zullo, F.}
\newblock New scattered linearized quadrinomials, 2024.

\bibitem{TimpanellaZini2024}
{\sc Timpanella, M., and Zini, G.}
\newblock On a family of linear {MRD} codes with parameters
  {$[8\times8,16,7]_q$}.
\newblock {\em Des. Codes Cryptogr. 92}, 3 (2024), 507--530.

\bibitem{zanella2019condition}
{\sc Zanella, C.}
\newblock A condition for scattered linearized polynomials involving {D}ickson
  matrices.
\newblock {\em J. Geom. 110}, 3 (2019), 1--9.

\end{thebibliography}

\end{document}